\documentclass[10pt, a4paper, reqno, english]{amsart}

\usepackage{amsfonts, amsthm, amsmath, amssymb, mathrsfs}
\usepackage{enumitem}
\usepackage[all]{xy} 
\usepackage{tikz}
\usepackage{mathtools}
\usetikzlibrary{matrix,arrows}

\theoremstyle{definition}
\newtheorem{defin}{Definition}
\newtheorem{example}[defin]{Example}

\theoremstyle{plain}
\newtheorem{theorem}[defin]{Theorem}
\newtheorem{prop}[defin]{Proposition}
\newtheorem{lemma}[defin]{Lemma}
\newtheorem{cor}[defin]{Corollary}
\newtheorem{construction}[defin]{Construction}

\theoremstyle{remark}
\newtheorem{remark}[defin]{Remark}

\numberwithin{defin}{section}
\numberwithin{equation}{section}

\newcommand{\CC}{\mathbb{C}} 

\newcommand{\RR}{\mathbb{R}} 

\newcommand{\QQ}{\mathbb{Q}} 

\newcommand{\ZZ}{\mathbb{Z}} 
\newcommand{\G}{\mathbb{G}} 
\newcommand{\A}{\mathbb{A}} 
\newcommand{\PP}{\mathbb{P}} 

\newcommand{\OO}{\mathcal{O}}

\newcommand{\field}{k} 
\newcommand{\fieldbar}{\overline{\field}} 
\newcommand{\kbar}{\fieldbar}
\newcommand{\fieldextension}{{\field'}}
\newcommand{\galois}{\mathfrak{g}} 
\newcommand{\galoiselement}{\text{\rm{g}}}
\newcommand{\galoiscocycle}{\sigma}
\newcommand{\compatible}{natural}

\newcommand{\group}{G} 
\newcommand{\groupbar}{\group_{\fieldbar}} 
\newcommand{\groupelement}{m}

\newcommand{\Xc}{X} 
\newcommand{\groupX}{\group_\Xc} 
\newcommand{\Xbar}{\Xcbar}
\newcommand{\Xcbar}{\Xc_{\fieldbar}} 

\newcommand{\Yc}{Y} 
\newcommand{\Ycbar}{\Yc_{\fieldbar}} 
\newcommand{\torsormor}{\pi} 
\newcommand{\torsormorbar}{\overline{\torsormor}} 
\newcommand{\cocycleY}{\beta}

\newcommand{\structuremor}{p}

\newcommand{\globalXbar}{\fieldbar[\Xcbar]}
\newcommand{\globalX}{\field[\Xc]}
\newcommand{\FfieldX}{\field(\Xc)}
\newcommand{\FfieldXbar}{\fieldbar(\Xcbar)}
\newcommand{\opensubset}{U}
\newcommand{\opensubsetbar}{\opensubset_{\fieldbar}}

\newcommand{\globalUbar}{\fieldbar[\opensubsetbar]}
\newcommand{\splitting}{\sigma}

\newcommand{\gradgroup}{M} 
\newcommand{\groupmorphism}{\varphi} 
\newcommand{\groupmorphismdual}{\widehat{\varphi}} 
\newcommand{\gradgroupdual}{\widehat{M}} 
\newcommand{\dualgroupelement}{h} 
\newcommand{\gradgroupdualX}{\gradgroupdual_{\Xc}} 
\newcommand{\typecox}{\lambda} 
\newcommand{\coxsheaf}{\mathcal{R}} 
\newcommand{\morphism}{\psi} 
\newcommand{\automorphism}{\psi} 
\newcommand{\coxring}{R} 
\newcommand{\coxisom}{\phi} 
\newcommand{\compconstant}{\alpha} 
\newcommand{\divisor}{D} 
\newcommand{\coxsec}{s} 
\DeclareMathOperator{\CaDiv}{CaDiv} 
\newcommand{\gradgroupdiv}{\gradgroup_{\typecox}} 
\DeclareMathOperator{\divi}{div} 

\newcommand{\freegroup}{\Lambda}
\newcommand{\presentation}{\varphi}
\newcommand{\freeelement}{L}
\newcommand{\kernfreegroup}{\Lambda_0}
\newcommand{\kernelement}{E}
\newcommand{\character}{\chi}
\newcommand{\freealgebra}{\mathcal{S}}
\newcommand\freeinvsheaf[1]{\mathcal{S}_{#1}}
\newcommand{\freeideal}{\mathcal{I}}
\newcommand{\freemor}{\pi}
\newcommand{\basis}{\mathcal{B}} 

\newcommand{\typecoxeff}{\typecox_{\eff}}
\newcommand{\gradgroupeff}{\gradgroup_{\eff}}
\newcommand{\gradgroupeffdual}{\widehat{\gradgroupeff}}
\newcommand{\freegroupeff}{\freegroup_{\eff}}
\newcommand{\basiseff}{\basis_{\eff}} 

\newcommand{\coxsheafbar}{\coxsheaf_{\fieldbar}}
\newcommand{\coxringbar}{\coxring_{\fieldbar}}
\newcommand{\twistedcoxsheaf}[1]{\coxsheaf^{#1}}

\newcommand{\gradgroupgenerator}{\mathcal{M}} 
\newcommand{\divisorsgenerator}{\mathcal{D}}
\newcommand{\gradgroupgeneratordiv}{\gradgroupgenerator_{\typecox}}

\newcommand{\freering}{S} 
\newcommand{\coxcoord}{\eta} 

\newcommand{\projectivemorphism}{\psi}
\newcommand{\galoisconstant}{\alpha}

\newcommand{\imaginary}{\text{\rm{i}}} 
\newcommand{\chat}{\Xc} 
\newcommand{\chatbar}{\chat_{\fieldbar}} 
\newcommand{\chatpoly}{P} 
\newcommand{\chatcoordX}{x} 
\newcommand{\chatcoordY}{y} 
\newcommand{\chatcoordZ}{z} 

\newcommand{\chatcoxbar}{\overline{R}} 
\newcommand{\chatcox}{R} 

\newcommand{\dcoordX}{x}

\newcommand{\ddiv}{D}

\newcommand{\dcoordfield}{\xi}

\newcommand{\Dfour}{{\mathbf D}_4}

\DeclareMathOperator{\identity}{id}
\DeclareMathOperator{\pic}{Pic}
\DeclareMathOperator{\Pic}{Pic}
\DeclareMathOperator{\Cl}{Cl}
\DeclareMathOperator{\spec}{Spec}

\DeclareMathOperator{\Supp}{Supp}

\DeclareMathOperator{\Hom}{Hom}
\DeclareMathOperator{\gal}{Gal}
\DeclareMathOperator{\type}{\texttt{type}}
\DeclareMathOperator{\eff}{eff}
\DeclareMathOperator{\Aut}{Aut}
\DeclareMathOperator{\weilres}{Res} 

\begin{document}

\title{Cox rings over nonclosed fields}

\author{Ulrich Derenthal} 

\address{Institut f\"ur Algebra, Zahlentheorie und Diskrete
  Mathematik, Leibniz Universit\"at Hannover,
  Welfengarten 1, 30167 Hannover, Germany}
\email{derenthal@math.uni-hannover.de}
\author{Marta Pieropan} 

\address{EPFL SB MATH CAG, B\^at. MA, Station 8, 1015 Lausanne, Switzerland}
\email{marta.pieropan@epfl.ch}
\date{September 3, 2018}


\setcounter{tocdepth}{1}

\begin{abstract}
  We give a definition of Cox rings and Cox sheaves for varieties over
  nonclosed fields that is compatible with torsors under quasitori,
  including universal torsors. We study their existence and
  classification, we make the relation to torsors precise, and we
  present arithmetic applications.
\end{abstract}

\maketitle

\tableofcontents

\section{Introduction}

Starting in the 1970s, Colliot-Th\'el\`ene and Sansuc studied torsors under
quasitori, i.e., under groups of multiplicative type. In
particular, they introduced \emph{universal torsors} of varieties with
finitely generated geometric Picard group, which were used to study
rational points and other arithmetic questions on geometrically
rational varieties over number fields; see \cite{MR899402,MR1845760}.

In the 1990s, Cox \cite{MR1299003} constructed homogeneous coordinate rings of
toric varieties over the field $\CC$ of complex numbers.  Salberger
\cite{MR1679841} used them in connection with universal torsors to study
rational points on split toric varieties over number fields.  Homogeneous
coordinate rings, or \emph{Cox rings}, were introduced for more general
classes of varieties by Hu, Keel \cite{MR1786494}, and others
\cite{MR2058459,MR1995130,MR2499353}.  Soon it was observed informally that
there should be a close connection to universal torsors (see the introductions
of \cite{MR2029863,MR2029868,MR2368955}).  So far, this connection has been
made precise only over algebraically closed fields
\cite[Theorem~5.6]{MR2498061}, \cite[\S 1.6]{arXiv:1003.4229}.

The problem of descending Cox rings to nonclosed fields was posed already in
\cite{MR2029863,MR2029868}.  The purpose of this article is to define Cox
rings and Cox sheaves for varieties over fields that are not necessarily
algebraically closed in a way that is compatible with the definition of
universal torsors and, more generally, of torsors under quasitori.

We start by revisiting and generalizing the definition and construction of Cox
rings and Cox sheaves over separably closed fields; see
Section~\ref{section:over_closed_fields}.  Our axiomatic definition allows us
to study their existence and classification over arbitrary fields, their
functoriality properties, and their precise relation to torsors; see
Section~\ref{section:over_nonclosed_fields}.  We conclude with some properties
of finitely generated Cox rings nd their arithmetic applications in
Sections~\ref{section:finitely_generated_cox_rings} and
\ref{sec:applications}, respectively.

It is an active field of research to determine generators and
relations of Cox rings over algebraically closed fields for many classes
of varieties. Since we also provide explicit methods to derive a
description of Cox rings and torsors over nonclosed fields from these,
we hope that our work lays the foundation for further arithmetic
applications of the vast literature on Cox rings.

\subsection{Motivation}\label{sec:motivation}

Torsors are central tools in the study of rational
points on a variety $X$ over a number field $k$. To study the Hasse principle
and weak approximation, often a \emph{local} description of torsors is
sufficient; see \cite[\S 3]{MR899402} and \cite[Part~2]{MR1845760}, for
example. Such a local description over sufficiently small open subsets $U$ of
$X$ can be determined as in \cite[\S 2.3]{MR899402}.

A \emph{global} description of torsors is needed for other applications, such
as to Manin's conjecture on the distribution of rational
points on Fano varieties \cite{MR974910}. In this application, 
initiated in \cite{MR1679841}, rational
points on a Fano variety $X$ are parameterized by integral points on its
universal torsors, and the necessary explicit global description of
universal torsors should be provided by Cox rings once their
generators and relations are known.

However, the relation between Cox rings and the parameterization of
rational points used in proofs of special cases of Manin's conjecture
is rarely made precise; of course the lack of a suitable definition of
Cox rings over nonclosed fields until now is a fundamental reason for
this. In many cases such as \cite{MR2332351}, a parameterization is
obtained by elementary manipulations of the equations defining $X$. If
$X$ is \emph{split} (i.e., the natural map $\pic(X) \hookrightarrow
\pic(\Xcbar)$ from the Picard group over $k$ to the geometric Picard group 
is an isomorphism), one typically observes that the
resulting auxiliary equations are the relations in a Cox ring of $X$,
at least over an algebraic closure $\kbar$ of $k$, and then one
expects that the parameterization is induced by a universal
torsor. For \emph{nonsplit} $X$, the geometric interpretation of the
parameterization sometimes remains completely unclear, as in
\cite{MR2373960}. In other cases such as \cite{MR2874644,MR3198752},
(universal) torsors are determined rigorously, while Cox rings remain
behind the scenes or are only briefly mentioned.

The only cases where the role of Cox rings is made rigorous seem to be
split toric varieties \cite{MR1679841} and split weak del Pezzo
surfaces \cite{arXiv:1312.6603}. For these classes of varieties, the
explicit description of Cox rings from
\cite{MR1299003,MR2029868,math.AG/0604194} is successfully applied to
prove many cases of Manin's conjecture; see
\cite{MR1679841,arXiv:1505.05789,arXiv:1312.6603} and the references
in \cite[\S 6.4.1]{arXiv:1003.4229}.

In this article, we develop a theory of Cox rings over nonclosed
fields that is suitable to handle torsor parameterizations of rational
points for not necessarily split varieties. In order to deal with
parameterizations by non-universal torsors, we introduce the notion of
Cox rings of a given \emph{type}, which encode a global description of
torsors under quasitori of the same \emph{type}.

\subsection{Cox rings of arbitrary type}\label{subsection:generalized_cox_rings}

In \cite{MR899402}, torsors under quasitori are classified in terms of their
\emph{type}.  Let $\field$ be a field with separable closure $\fieldbar$ and
Galois group $\galois=\gal(\fieldbar/\field)$. We recall that there is an
antiequivalence of categories between quasitori over $\field$ (i.e., smooth
groups of multiplicative type of finite type) and finitely generated
$\galois$-modules (i.e., $\galois$-modules that are finitely generated as
abelian groups) with torsion coprime to the characteristic of $\field$, which
sends a quasitorus $G$ to its group of characters
$M\coloneqq \Hom_{\fieldbar}(\groupbar,\G_{m,\fieldbar})$.

Let $X$ be a variety over $\field$ such that $\Xbar$ has only constant
invertible regular functions.  Torsors over $X$ under a quasitorus $G$
with group of characters $M$ are classified by the middle term of the
exact sequence
\begin{equation}\label{eq:exact_sequence_type}
  0\rightarrow H^1_{\text{\it{\'et}}}(\field,\group) \rightarrow
  H^1_{\text{\it{\'et}}}(\Xc,\groupX)\stackrel{\type}
  {\longrightarrow} 
  \Hom_{\galois}(M,\pic(\Xcbar))
  \rightarrow H^2_{\text{\it{\'et}}}(\field,\group),
\end{equation}
from \cite[Th\'eor\`eme 1.5.1]{MR899402}. The map $\type$ sends the
isomorphism class of a torsor $Y$ to its \emph{type}, namely to the
homomorphism $\lambda\colon M \to \Pic(\Xcbar)$ of $\galois$-modules that
sends a character $\chi\in M$ to the isomorphism class of the
pushforward $\chi_*Y$ in
$H^1_{\text{\it{\'et}}}(\Xcbar,\G_{m})=\pic(\Xcbar)$; see \cite[\S
2.0]{MR899402}. By definition, a universal torsor is a torsor of
\emph{identity type} (i.e., $\gradgroup=\pic(\Xcbar)$ and
$\lambda=\identity_{\pic(\Xcbar)}$).

In Definitions~\ref{def:cox_ring} and \ref{def:cox_sheaf_nonclosed}, we
introduce the notion of Cox rings of a given \emph{type}
\begin{equation*}
  \lambda\colon \gradgroup\to\pic(\Xcbar),
\end{equation*}
where $\gradgroup$ is a finitely
generated $\galois$-module whose torsion is coprime to the characteristic of
$\field$.  This is a $k$-algebra $R$ with an $M$-grading on $R_{\fieldbar}$
such that, for every $m \in M$, its degree-$m$-part is isomorphic to
$H^0(X,\OO(D))$ for a divisor $D$ of class $\lambda(m)$; the main difficulty
is to define a multiplication.  In particular, a Cox ring of identity type is
a Cox ring in the original sense.

Below, our results are all stated for Cox rings and Cox sheaves of
arbitrary type $\lambda$. The statements for ordinary Cox rings and
universal torsors are then recovered by taking the identity type
$\gradgroup\coloneqq \pic(\Xcbar)$ and $\lambda\coloneqq \identity_{\pic(\Xcbar)}$.

\subsection{Cox rings over closed fields}\label{sec:intro_def}

In the literature, a Cox ring (of identity type) of a variety $X$ over
an algebraically closed field $k$ is a $\pic(X)$-graded $k$-algebra of
the form
\begin{equation}\label{eq:cox_informal}
  R = \bigoplus_{\Pic(X)} H^0(X, \OO(D)),
\end{equation}
where $D$ runs through a set of divisors forming a system of representatives
of $\Pic(X)$. While the structure of a $\Pic(X)$-graded $k$-vector space is
clear, there is no canonical way to define a multiplication on $R$ that turns
it into a graded $k$-algebra.

If $X$ is smooth and projective with free Picard group, there is a well-known
construction of Cox rings (see \cite{MR2029868,MR2278756,MR2498061,MR2824848},
for example) that depends on the choice of divisors whose classes form a
$\ZZ$-basis of $\Pic(X)$.  It is important to note that a different choice of
divisors leads to a Cox ring that is \emph{noncanonically} isomorphic to the
previous one; this makes descending Cox rings an interesting problem (see
\cite{MR2029863,MR2029868}).

When the Picard group is not free, one can construct Cox rings by
choosing divisors whose classes generate $\Pic(X)$ together with a
\emph{shifting family} or an \emph{identifying character}; see
\cite{MR1995130}, \cite[\S 1.4.2]{arXiv:1003.4229}.
Another option is to work with the choice of a $k$-rational
point on $X$ as in \cite[Definition~1.1]{MR2029863} and
\cite[Construction~1.4.2.3]{arXiv:1003.4229}.

In our approach to Cox rings over closed fields (see
Definition~\ref{def:cox_ring}), two aspects are new: The first novelty is that
our definition is axiomatic, while the previously existing definitions can be
interpreted as constructions; this allows us to determine automorphisms, to
study Galois descent to nonclosed fields, and to phrase and answer existence
and uniqueness questions. Our axioms emphasize the role of the divisors
associated with every homogeneous element of a Cox ring. The second novelty is
that we define Cox rings of arbitrary \emph{type}.

Then we adapt the previous constructions of Cox rings to the case of
arbitrary type (see Construction
\ref{construction:cox_sheaf_abstract}) and we show that our axioms are
satisfied precisely by the rings obtained from such constructions.  
In the proof of Theorem~\ref{thm:splttings} below, we present several
other constructions of Cox rings that seem to be new even over
closed fields and are inspired by the theory of torsors.

We recall that one can also construct Cox rings graded by the Weil
divisor class group $\Cl(X)$ instead of $\pic(X)$; see
\cite{MR2058459,arXiv:1003.4229}.  If $X$ is singular, $\Cl(X)$ may
differ from $\pic(X)$, and then these Cox rings correspond to
\emph{good quotients} $Y\to X$ that are not torsors; see
\cite[\S4.2.1]{arXiv:1003.4229}. Hence we only consider
$\pic(X)$-graded Cox rings.

Some articles such as \cite{MR1786494,MR3275656,MR3352043} study Cox
rings defined via a choice of divisors forming a $\QQ$-basis of
$\Pic(X) \otimes_\ZZ \QQ$ (or $\Cl(X) \otimes_\ZZ \QQ$). Here,
different choices generally lead to nonisomorphic results.  This is
enough to study finite generation of the resulting rings, but it is
clearly not suitable to investigate torsors.

\subsection{The definition over nonclosed fields}
We define Cox rings of varieties over nonclosed fields by Galois descent.
Since Cox rings over separably closed fields may have nontrivial automorphisms 
(see Proposition \ref{prop:cox_sheaf_aut}), Galois descent to nonclosed fields 
is a nontrivial task, if possible at all.

Let $X$ be a variety over a field $k$ with separable closure $\overline k$.
In Section \ref{section:over_nonclosed_fields}, we define a Cox ring of $X$ of
type $\lambda$ to be a $k$-form $R$ of a Cox ring of $\Xcbar$ of type
$\lambda$ such that the action of the Galois group
$\galois \coloneqq  \gal(\fieldbar/\field)$ on $\coxring \otimes_\field \fieldbar$ is
compatible with the Cox ring structure in the sense of our Definition
\ref{def:galois_action}.  This definition, together with the axiomatic
definition of Cox rings over separably closed fields, enables us to study the
classification and existence of Cox rings over nonclosed fields.

\subsection{Classification and existence}

Our first main result concerns the relation between Cox rings, Cox sheaves 
and torsors under quasitori, and yields the classification of Cox rings over 
nonclosed fields.

\begin{theorem}\label{introduction:theorem:classification}
  Let $k$ be a field. Let $\galois\coloneqq \gal(\kbar/k)$, where $\kbar$ is a
  separable closure of $k$.  Let $X$ be a $\field$-variety with
  $\globalXbar^\times=\kbar^\times$. Let $\gradgroup$ be a finitely generated
  $\galois$-module whose torsion is coprime to the characteristic of $k$, and
  $\typecox\colon \gradgroup\to\Pic(\Xbar)$ a $\galois$-equivariant homomorphism.
  
  The contravariant functor
  \begin{align*}
    \big\{\text{Cox sheaves of $\Xc$ of type $\typecox$}\big\}
    & \longrightarrow \big\{\text{$\Xc$-torsors of type $\typecox$}\big\}\\
    \coxsheaf\quad
    & \longmapsto\quad \spec_{\Xc}\coxsheaf
  \end{align*}
  is an anti-equivalence of categories.

  If the image of $\typecox$ is generated by effective divisor
  classes, then the covariant functor
  \begin{align*}
    \big\{\text{Cox sheaves of $\Xc$ of type $\typecox$}\big\}
    &\longrightarrow \big\{\text{Cox rings of $\Xc$ of type $\typecox$}\big\}\\
    \coxsheaf\quad
    & \longmapsto\quad \coxsheaf(\Xc)
  \end{align*}
  is an equivalence of categories.
\end{theorem}

Conversely, a torsor $\pi\colon  Y \to X$ of type $\typecox$ corresponds to
the Cox sheaf $\pi_*\OO_Y$ and to the Cox ring $\OO_Y(Y)$ of the same
type.  Given a Cox ring, it is more involved to recover the
corresponding Cox sheaf and torsor; see
Sections~\ref{section:over_nonclosed_fields} and
\ref{section:finitely_generated_cox_rings}.

Because of the existence of nontrivial automorphisms, a variety $X$ over a
nonclosed field $k$ might have no Cox ring (see
Example~\ref{exa:nonexistence}), or it might have several nonisomorphic Cox
rings of the same type, which can be obtained from one such Cox ring by
twisting by cocycles. 

Our second main result gives existence criteria for Cox rings:

\begin{theorem}\label{thm:splttings}
  Let $k$ be a field. Let $\galois\coloneqq \gal(\kbar/k)$, where $\kbar$ is a
  separable closure of $k$.  Let $X$ be a $\field$-variety with
  $\globalXbar^\times=\kbar^\times$. Let $\gradgroup$ be a finitely generated
  $\galois$-module whose torsion is coprime to the characteristic of $k$, and
  $\typecox\colon \gradgroup\to\Pic(\Xbar)$ a $\galois$-equivariant homomorphism.
  \begin{enumerate}[label=(\Alph*), ref={\it{(\Alph*)}}]
  \item\label{it:character} Given two finite $\galois$-invariant sets
    $\gradgroupgenerator \subseteq M$ and
    $\divisorsgenerator \subseteq \CaDiv(\Xbar)$ such that $\gradgroupgenerator$
    generates $M$ and $\lambda(M)=\{[D]\colon D\in\divisorsgenerator\}$, we consider
    the projections
    \begin{equation*}
      \varphi_{\gradgroupgenerator}\colon  \bigoplus_{\substack{(m,D) \in
          \gradgroupgenerator\times\divisorsgenerator\\
          \typecox(\groupelement)=[\divisor]}}\ZZ(m,D) \to \gradgroup,
      \qquad
      \varphi_{\divisorsgenerator}\colon \bigoplus_{\substack{(m,D) \in
          \gradgroupgenerator\times\divisorsgenerator\\
          \typecox(\groupelement)=[\divisor]}}\ZZ(m,D) \to \CaDiv(\Xbar).
    \end{equation*}
    
    Then Cox sheaves and torsors of type $\lambda$ over $X$ exist if and only if,
    for one (equivalently, for every) pair
    $\gradgroupgenerator,\divisorsgenerator$ as above, there exists a
    $\galois$-equivariant group homomorphism
    \begin{equation*}
      \chi\colon \ker(\varphi_{\gradgroupgenerator})\to \kbar(X)^\times
    \end{equation*}
    such that the principal divisor defined by
    $\chi(x)$ is $\varphi_{\divisorsgenerator}(x)$ for all 
    $x\in\ker(\varphi_{\gradgroupgenerator})$.
     
  \item\label{it:splittings} For the existence of Cox sheaves, Cox rings and
    torsors of type $\typecox$ over $X$, the existence of one of the following
    suffices.
    \begin{enumerate}[label=(\arabic*), ref={\it{(\arabic{*})}}]
    \item\label{it:splitting_kX} A $\galois$-equivariant splitting of the
      natural exact sequence of $\galois$-modules
      \begin{equation*}
        1 \to \kbar^\times \to \kbar(\Xbar)^\times \to \kbar(\Xbar)^\times/\kbar^\times \to 1.
      \end{equation*}
    \item\label{it:splitting_kU} An open subset $U \subseteq X$ 
      such that $\typecox(\gradgroup)$ is generated by classes of Cartier divisors on
      $\Xbar$ supported outside $U_{\kbar}$, with a $\galois$-equivariant
      splitting of the natural exact sequence of $\galois$-modules
      \begin{equation*}
        1 \to \kbar^\times \to \kbar[U_{\kbar}]^\times
        \to \kbar[U_{\kbar}]^\times/\kbar^\times \to 1.
      \end{equation*}
    \item\label{it:rational_point} A $k$-rational point on $X$.
    \end{enumerate}
    
    Conversely, if $\Xcbar$ is locally factorial and $\typecox$ is
    injective, then the existence of splittings as in
    (\ref{it:splitting_kU}) for all suitable $U$ is necessary for the
    existence of Cox sheaves or torsors of $X$ of type $\typecox$.  If
    $k$ is perfect and $\Xcbar$ is locally factorial, then the existence
    of a splitting as in (\ref{it:splitting_kX}) is necessary for the
    existence of Cox sheaves of type $\identity_{\pic(\Xcbar)}$ and
    universal torsors of $X$.
  \item\label{it:cox_sheaves_rings} If $\typecox(\gradgroup)$ is generated by effective divisor classes,
    then the existence of Cox sheaves of $X$ of type $\typecox$ is necessary for
    the existence of Cox rings of $\Xc$ of type $\typecox$.
  \end{enumerate}
\end{theorem}

Part \ref{it:character} is based on a Galois-equivariant version of
Construction~\ref{construction:cox_sheaf_abstract} combined with
Galois descent.  If $X$ is smooth, then part \ref{it:splittings} follows from the
existence criteria for torsors from \cite[\S 2.2]{MR899402} via
Theorem~\ref{introduction:theorem:classification} ; in particular, the
criterion \ref{it:splittings}\ref{it:splitting_kX} is the vanishing of the \emph{elementary obstruction}
\cite[Definition~2.2.1]{MR899402}.  In our proof, however, we take the
criteria in \cite[\S 2.2]{MR899402} only as guidelines for new
Galois-equivariant constructions of Cox rings over $\kbar$; then
Galois descent gives Cox rings over $k$ (see
Constructions~\ref{construction:cox_sheaf_splitting_U} and
\ref{construction:rational_point_cox_sheaf}). This leads to a proof of
the existence criteria \ref{it:splittings} for Cox rings that avoids the machinery
from \cite{MR899402} and also applies to singular varieties; then
Theorem~\ref{introduction:theorem:classification} transfers the
existence criteria to torsors in this more general setting.

Note that the existence of a $k$-rational point is not necessary for the
existence of Cox rings and Cox sheaves on $X$; see
Example~\ref{exa:cox_ring_no_rational_point}.

An interpretation of universal torsors in terms of gerbes is given in
\cite[\S 13]{arXiv:1610.07341}.

\subsection{Applications}

In the last sections, we study Cox rings that are finitely generated as
$k$-algebras because such a Cox ring $R$ gives a global description of
the corresponding torsor as an open subset of the affine variety
$\spec R$.  We recall that there is a vast and growing literature
computing finitely generated Cox rings of identity type in terms of
generators and relations over algebraically closed fields. 

In Section \ref{section:finitely_generated_cox_rings}, we make precise how a
Cox ring can be described by generators and relations (note that it is not
enough to describe divisors corresponding to the generators, but that one
needs a suitable \emph{character} to obtain a well-defined object). We show
that a Cox ring $R'$ of arbitrary type can be obtained as a pullback of a Cox
ring $R$ of identity type, and we explain a method to produce generators and
relations of $R'$ from those of $R$ if the latter is finitely generated.

In Section \ref{sec:applications}, we we give examples where the
parameterizations appearing in proofs of Manin's conjecture for some nonsplit
varieties \cite{MR2373960,MR2874644} can be interpreted in terms of Cox rings
of type $\pic(\Xc)\hookrightarrow\Pic(\Xcbar)$.  Recent work of Destagnol
\cite{arXiv:1509.07060,arXiv:1711.01882} applies our techniques to compute Cox
rings of a different injective type to the computation of Peyre's constant in
the asymptotic formula predicted by Manin's conjecture for some Ch\^atelet
surfaces. In particular, our theory allows to complete the comparison of the
leading constant in \cite{MR3103132} with Peyre's constant (see \cite[\S
5.2]{arXiv:1711.01882}).

For algebraic-geometric applications of our theory, see \cite{MR3531747} and
\cite{arXiv:1708.02797}.

\subsection*{Acknowledgements}

The authors were supported by grants DE 1646/3-1 and ES 60/10-1 of the
Deutsche Forschungsgemeinschaft.  We thank J.-L.~Colliot-Th\'el\`ene,
D.~A.~Cox, J.~Hausen, A.~N.~Skorobogatov, and the referees for useful
remarks and discussions, and we thank S.~Keicher for explaining the
use of the Maple package \cite{software_Cox_rings} to us.
Example~\ref{exa:quintic_dp} is based on a discussion with
D.~Loughran.

\section{Over separably closed fields}\label{section:over_closed_fields}

We start by fixing our setting and some notation.  In this section, we work
over a separably closed field $\field$.  Let $\Xc$ be an integral
$\field$-variety.  We always assume that $\Xc$ has only constant invertible
regular functions (i.e., $\globalX^\times=\field^\times$, where $\globalX$
denotes the ring of global sections $H^0(\Xc,\OO_\Xc)$ of the structure sheaf
of $\Xc$).

Let $\CaDiv(\Xc)$ be the group of Cartier divisors of $\Xc$. For every divisor
$\divisor\in\CaDiv(\Xc)$, we denote by $[\divisor]$ its class in $\Pic(\Xc)$.
Given a Cartier divisor $\divisor=\{(U_i,f_i)\}_i$, its associated sheaf
$\OO_\Xc(\divisor)$ is the invertible sheaf obtained by gluing
$f_i^{-1}\OO_{U_i}$.  This is a subsheaf of the constant sheaf associated with
the function field $\FfieldX$ of $\Xc$. Therefore, given two Cartier divisors
$\divisor_1,\divisor_2$ and sections
$\coxsec_i\in H^0(U,\OO_\Xc(\divisor_i))$, $i\in\{1,2\}$, the product
$\coxsec_1\coxsec_2\in H^0(U,\OO_\Xc(\divisor_1+\divisor_2))$ is well-defined.
Given a Cartier divisor $\divisor=\{(U_i,f_i)\}_i$ and an element
$f\in \FfieldX^\times$, we denote by $\divi_\divisor(f)$ the Cartier divisor
$\{(U_i,ff_i)\}_i$. Note that $f \in H^0(\Xc,\OO(\divisor))$ if and only if
$\divi_\divisor(f)$ is an effective divisor.  The support of a Cartier divisor
$D=\{(U_i,f_i)\}_i$ is
\begin{equation*}
  \Supp(D)\coloneqq \{x\in X\colon  f_i\notin\OO_{X,x}^\times \text{ if } x\in U_i\}.
\end{equation*}

We fix a finitely generated abelian group $\gradgroup$ whose torsion
is coprime to the characteristic of $\field$ (i.e.,~the characteristic
of $\field$ does not divide the order of its torsion subgroup), and a
group homomorphism $\typecox\colon \gradgroup\to\pic(\Xc)$, which will be
the \emph{type} of our Cox rings, Cox sheaves and torsors. We define
\begin{equation*}
  \gradgroupdiv\coloneqq \{(\groupelement,\divisor)\in
  \gradgroup\times\CaDiv(\Xc):  [\divisor]=\typecox(\groupelement)\}.
\end{equation*}
We denote by $\gradgroupdual\coloneqq \spec\field[\gradgroup]$ the quasitorus
dual to $\gradgroup$ under the antiequivalence of categories recalled
in Section~\ref{subsection:generalized_cox_rings}.

\begin{defin}\label{def:cox_ring}
  A \emph{Cox ring of $\Xc$ of type $\typecox$} is an $\gradgroup$-graded
  $\field$-algebra $\coxring$ together with a map
  \begin{equation*}
    \divi\colon \bigcup_{\groupelement\in\gradgroup}(\coxring_{\groupelement}\smallsetminus\{0\})
    \to\CaDiv(\Xc),
  \end{equation*}
  where $\coxring_{\groupelement}$ denotes the degree-$\groupelement$-part of
  $\coxring$, such that, for each $(\groupelement,\divisor)\in\gradgroupdiv$,
  there exists a $\field$-vector space isomorphism
  $\coxisom_{\groupelement, \divisor}\colon \coxring_{\groupelement}\to
  H^0(\Xc,\OO_\Xc(\divisor))$ satisfying
  $\divi(\coxsec)=\divi_{\divisor}(\coxisom_{\groupelement,
    \divisor}(\coxsec))$ for all nonzero $\coxsec\in\coxring_\groupelement$,
  and $\divi(\coxsec_1\coxsec_2)=\divi(\coxsec_1)+\divi(\coxsec_2)$ for all
  nonzero homogeneous elements $\coxsec_1$, $\coxsec_2$ of $\coxring$.

  Let $\coxring$ and $\coxring'$ be two Cox rings of $\Xc$ of type
  $\typecox$. A morphism $\morphism\colon \coxring\to\coxring'$ of $\gradgroup$-graded $\field$-algebras
  is a \emph{morphism of Cox rings of $\Xc$
    of type $\typecox$} if $\divi'(\morphism(s))=\divi(s)$ for all nonzero
  homogeneous $s\in\coxring$, where $\divi$ and $\divi'$ are the maps
  associated with $\coxring$ and $\coxring'$, respectively.

  A \emph{structure of Cox ring of $\Xc$ of type $\typecox$} on a
  $\field$-algebra $\coxring$ consists of an $\gradgroup$-grading on
  $\coxring$ and a map $\divi$ as above.
\end{defin}

Note that we only require the existence of the isomorphisms
$\coxisom_{\groupelement, \divisor}$; they are not part of the data. We will
see in Remark~\ref{rem:up_to_constant} that they are determined uniquely up to
nonzero constants by the map $\divi$ if they exist.

In \cite{MR1995130,arXiv:1003.4229}, a Cox ring is the ring of global
sections of a Cox sheaf. Now we define Cox sheaves of type $\typecox$. In
Proposition~\ref{prop:cox_ring_cox_sheaf}, we prove that the Cox rings in
Definition \ref{def:cox_ring} are the rings of global sections of the
Cox sheaves defined below.

\begin{defin}\label{def:cox_sheaf}
  A \emph{Cox sheaf of $\Xc$ of type $\typecox$} is a sheaf $\coxsheaf$ of $\gradgroup$-graded $\OO_\Xc$-algebras together with a family of isomorphisms of $\OO_\Xc$-modules 
  \begin{equation*}
    \left\{\coxisom_{\groupelement, \divisor}\colon 
      \coxsheaf_{\groupelement}\to
      \OO_\Xc(\divisor)\right\}_{(\groupelement,\divisor)\in\gradgroupdiv},
  \end{equation*}
  where $\coxsheaf_{\groupelement}$ denotes the degree-$\groupelement$-part of
  $\coxsheaf$, such that, for every $(\groupelement_1,\divisor_1)$,
  $(\groupelement_2,\divisor_2)\in\gradgroupdiv$, there exists a nonzero
  constant $\compconstant\in\field$ that satisfies
  \begin{equation*}
    \coxisom_{\groupelement_1, \divisor_1}(\coxsec_1)
    \coxisom_{\groupelement_2, \divisor_2}(\coxsec_2)=
    \compconstant\coxisom_{\groupelement_1+\groupelement_2, \divisor_1+\divisor_2}(\coxsec_1\coxsec_2)
  \end{equation*}
  for all $\coxsec_i\in \coxsheaf_{\groupelement_i}(U)$, $i\in\{1,2\}$, and
  all open subsets $U$ of $\Xc$.

  A \emph{morphism of Cox sheaves of $\Xc$ of type $\typecox$} is a morphism
  of $\gradgroup$-graded $\OO_\Xc$-algebras.

  A \emph{structure of Cox sheaf of $\Xc$ of type $\typecox$} on a sheaf
  $\coxsheaf$ of $\OO_\Xc$-algebras consists of an $\gradgroup$-grading on
  $\coxsheaf$ and a family of isomorphisms
  $\{\coxisom_{\groupelement,
    \divisor}\}_{(\groupelement,\divisor)\in\gradgroupdiv}$ as above.
\end{defin} 

\begin{prop}\label{prop:equivalent_definition_cox_ring_div}
  An $\gradgroup$-graded $\field$-algebra $\coxring$ has a structure of Cox
  ring of $\Xc$ of type $\typecox$ if and only if there exists a family of
  isomorphisms of $\field$-vector spaces
  \begin{equation*}
    \{\coxisom_{\groupelement, \divisor}\colon \coxring_{\groupelement}\to
    H^0(\Xc,\OO_\Xc(\divisor))\}_{(\groupelement,\divisor)\in\gradgroupdiv},
  \end{equation*}
  such that, for every
  $(\groupelement_1,\divisor_1),(\groupelement_2,\divisor_2)\in\gradgroupdiv$,
  there exists a nonzero constant $\compconstant\in\field$ that satisfies
 \begin{equation*}
   \coxisom_{\groupelement_1, \divisor_1}(\coxsec_1)
   \coxisom_{\groupelement_2, \divisor_2}(\coxsec_2)=
   \compconstant\coxisom_{\groupelement_1+\groupelement_2, \divisor_1+\divisor_2}(\coxsec_1\coxsec_2)
 \end{equation*}
  for all $\coxsec_i\in \coxring_{\groupelement_i}$, $i\in\{1,2\}$.
  
  In particular, if $\coxsheaf$ is a Cox sheaf of $X$ of type $\lambda$, then
  $\coxsheaf(X)$ is a Cox ring of $X$ of type $\lambda$.
\end{prop}

\begin{proof}
  If $\coxring$ is an $\gradgroup$-graded $\field$-algebra endowed with a
  family of isomorphisms $\coxisom_{\groupelement, \divisor}$ as in the
  statement, for every $\groupelement\in\gradgroup$, choose a divisor
  $\divisor\in\CaDiv(\Xc)$ such that
  $(\groupelement,\divisor)\in\gradgroupdiv$ and define
  $\divi(\coxsec)\coloneqq \divi_{\divisor}(\coxisom_{\groupelement,
    \divisor}(\coxsec))$ for all nonzero $\coxsec\in\coxring_{\groupelement}$.

  For the converse implication, assume that $\coxring$ is a Cox ring of $\Xc$
  of type $\typecox$ with a map $\divi$. We choose isomorphisms
  $\coxisom_{\groupelement,\divisor}$ as in Definition \ref{def:cox_ring}. We
  prove that the family of isomorphisms $\coxisom_{\groupelement,\divisor}$
  defines on $\coxring$ a structure of Cox ring of $\Xc$ of type $\typecox$ as
  in the statement.  Let
  $(\groupelement_1,\divisor_1),(\groupelement_2,\divisor_2)\in\gradgroupdiv$. If
  $\coxring_{\groupelement_1}=0$ or $\coxring_{\groupelement_2}=0$, let
  $\alpha\coloneqq 1$. If
  $\coxsec_i\in\coxring_{\groupelement_i}\smallsetminus\{0\}$, $i\in\{1,2\}$,
  the elements
  \begin{equation*}
    \alpha\coloneqq 
    \coxisom_{\groupelement_1, \divisor_1}(\coxsec_1)
    \coxisom_{\groupelement_2, \divisor_2}(\coxsec_2)
    \coxisom_{\groupelement_1+\groupelement_2, \divisor_1+\divisor_2}(\coxsec_1\coxsec_2)^{-1}
  \end{equation*}
  belong to $\field^\times$ because
  \begin{equation*}
    \divi_{\divisor_1+\divisor_2}(\coxisom_{\groupelement_1+\groupelement_2, \divisor_1+\divisor_2}(\coxsec_1\coxsec_2))=\divi_{\divisor_1}(\coxisom_{\groupelement_1, \divisor_1}(\coxsec_1))+\divi_{\divisor_2}(\coxisom_{\groupelement_2, \divisor_2}(\coxsec_2)),
  \end{equation*}
  and do not depend on the chosen sections
  $\coxsec_i\in\coxring_{\groupelement_i}\smallsetminus\{0\}$ because the
  morphisms $\coxisom_{\groupelement,\divisor}$ are linear.
\end{proof}

We observe that a structure of Cox ring or of Cox sheaf is actually determined
by a subfamily of isomorphisms $\coxisom_{\groupelement,\divisor}$. This is
made precise via the following lemma, which characterizes the morphisms of
$\OO_\Xc$-modules between invertible sheaves.

\begin{lemma}\label{lem:isom_inv_sheaves}
  Let $D,D'$ be two Cartier divisors and $\morphism\colon \OO_\Xc(D)\to\OO_X(D')$ a
  morphism of $\OO_X$-modules. Then there exists an $f\in \FfieldX$
  such that $\morphism(s)=fs$ for all $s\in
  H^0(U,\OO_X(D))$ and all open subsets $U\subseteq X$. If $D=D'$, then
  $f\in \globalX$. If $\morphism$ is an isomorphism, then $D=\divi_{D'}(f)$.
\end{lemma}

\begin{proof}
  Without loss of generality, we can assume that $D$ and $D'$ are trivialized
  by the same open covering of $X$, say $D=\{(U_i,f_i)\}_i$ and
  $D'=\{(U_i,f'_i)\}_i$.  Let $s\in H^0(U_i,\OO_X(D))$.  Then
  $\morphism(s)=\morphism(f_i^{-1})f_is$, and
  $\morphism(f_i^{-1})=u_i{f_i'}^{-1}$ for some $u_i\in H^0(U_i,\OO_X)$.
  Moreover, $u_i{f_i'}^{-1}f_i=u_j{f_j'}^{-1}f_j$ in $\FfieldX$ for all $i$
  and $j$ by restricting $\morphism$ to $U_i\cap U_j$. Take
  $f=u_i{f_i'}^{-1}f_i$.

  If $D=D'$, then $f\in H^0(U_i,\OO_X)$ for all
  $i$, hence $f\in \globalX$. If $\morphism$ is surjective, then
  $u_i\in H^0(U_i,\OO_X)^\times$ is a unit for all $i$.
\end{proof}

\begin{remark}\label{rem:up_to_constant}
  The isomorphisms $\coxisom_{\groupelement,\divisor}$ in Definition
  \ref{def:cox_sheaf} and Proposition
  \ref{prop:equivalent_definition_cox_ring_div} are uniquely determined up to
  multiplication by nonzero elements of $\field$. Therefore, given a Cox sheaf
  $\coxsheaf$ of $\Xc$ of type $\typecox$, we say that two families
  $\{\coxisom_{\groupelement,\divisor}\}_{(\groupelement,\divisor)\in\gradgroupdiv}$
  and
  $\{\coxisom'_{\groupelement,\divisor}\}_{(\groupelement,\divisor)\in\gradgroupdiv}$
  define the same Cox sheaf structure on $\coxsheaf$ if,
  for every $(\groupelement,\divisor)\in\gradgroupdiv$, there is
  $\compconstant_{\groupelement,\divisor}\in\field^\times$ such that
  $\coxisom'_{\groupelement,\divisor}=
  \compconstant_{\groupelement,\divisor}\coxisom_{\groupelement,\divisor}$.

  By Lemma \ref{lem:isom_inv_sheaves}, a Cox sheaf structure on
  $\coxsheaf$ is determined once, for each $\groupelement\in\gradgroup$, the isomorphism
  $\coxisom_{\groupelement,\divisor}$ is known for one
  $(\groupelement,\divisor)\in\gradgroupdiv$.

  Analogous remarks apply to Cox rings.
\end{remark}

\begin{remark}\label{rem:equivalent_characterization_cox_ring}
  Using the equivalent characterization of Proposition
  \ref{prop:equivalent_definition_cox_ring_div}, a morphism of
  $\gradgroup$-graded $\field$-algebras $\morphism\colon \coxring\to\coxring'$
  between two Cox rings of $\Xc$ of type $\typecox$ is a \emph{morphism of Cox
    rings of $\Xc$ of type $\typecox$} as in Definition \ref{def:cox_ring} if
  and only if it is compatible with the families of isomorphisms
  $\{\coxisom_{\groupelement,
    \divisor}\}_{(\groupelement,\divisor)\in\gradgroupdiv}$ and
  $\{\coxisom'_{\groupelement,
    \divisor}\}_{(\groupelement,\divisor)\in\gradgroupdiv}$ associated with
  $\coxring$ and $\coxring'$, respectively, as follows: for every
  $(\groupelement,\divisor)\in\gradgroupdiv$, there is
  $\compconstant_{\groupelement,\divisor}\in\field^\times$ such that
  $\coxisom'_{\groupelement,
    \divisor}\circ\morphism\circ\coxisom_{\groupelement, \divisor}^{-1}=
  \compconstant_{\groupelement,\divisor}\identity_{H^0(\Xc,\OO_{\Xc}(\divisor))}$.
  Moreover, a \emph{structure of Cox ring of $\Xc$ of type $\typecox$} on a
  $\field$-algebra $\coxring$ is an $\gradgroup$-grading on $\coxring$ and a
  family of isomorphisms
  $\{\coxisom_{\groupelement,
    \divisor}\}_{(\groupelement,\divisor)\in\gradgroupdiv}$ as in Proposition
  \ref{prop:equivalent_definition_cox_ring_div}.
\end{remark}

Our axiomatic definition of Cox sheaves is compatible with the
previous definitions of $\Pic(\Xc)$-graded Cox sheaves as follows. The
sheaves defined in \cite[Lemma 3.5]{MR1995130}, with the family of
isomorphisms in \cite[Lemma 3.5(ii)]{MR1995130}, are Cox sheaves of
identity type in the sense of Definition \ref{def:cox_sheaf}.  If
$\Xc$ is locally factorial, the Cox sheaves constructed in
\cite[Construction 1.4.2.1]{arXiv:1003.4229}, with the family of
isomorphisms in \cite[Lemma 1.4.3.4]{arXiv:1003.4229}, are Cox sheaves
of identity type in the sense of Definition \ref{def:cox_sheaf}. 

The following construction is analogous to \cite[\S 3]{MR1995130}, \cite[\S
2]{MR2499353}, and \cite[Construction~1.4.2.1]{arXiv:1003.4229}.  In
Proposition \ref{prop:isomorphism_cox_sheaves_abstract}, we will show that
every Cox sheaf of $\Xc$ of type $\typecox$ is isomorphic to one defined in
Construction \ref{construction:cox_sheaf_abstract} and that the Cox sheaves
obtained from Construction \ref{construction:cox_sheaf_abstract} are all
isomorphic.

\begin{construction}\label{construction:cox_sheaf_abstract}
  Let 
  \begin{equation*}
    0\to\kernfreegroup\to\freegroup\stackrel{\presentation}{\to}\gradgroup\to 0
  \end{equation*}
  be a presentation of $\gradgroup$ by a finitely generated free group
  $\freegroup$. Let $\freeelement_1, \dots,\freeelement_n$ be a basis of
  $\freegroup$. For $i \in \{1, \dots, n\}$, let $\divisor_i$ be a Cartier divisor
  representing the class $\typecox(\presentation(\freeelement_i))$ in
  $\pic(\Xc)$. For every $\freeelement=\sum_{i=1}^na_i\freeelement_i$ in
  $\freegroup$, let
  $\freeinvsheaf{\freeelement}\coloneqq \OO_\Xc(\sum_{i=1}^na_i\divisor_i)$.  Endow
  $\freealgebra\coloneqq \bigoplus_{\freeelement\in\freegroup}\freeinvsheaf{\freeelement}$
  with the multiplication of sections induced by the embeddings
  $\freeinvsheaf{\freeelement}\subseteq\FfieldX$. Let
  $\character\colon \kernfreegroup\to\FfieldX^\times$ be a homomorphism of groups
  that satisfies $\divi_0(\character(\kernelement))=\sum_{i=1}^na_i\divisor_i$
  for all $\kernelement=\sum_{i=1}^na_i\freeelement_i\in\kernfreegroup$. We
  call $\character$ a \emph{character associated with} $\freealgebra$. Let
  $\freeideal$ be the sheaf of ideals of $\freealgebra$ locally generated by
  the sections $1-\character(\kernelement)$ for all
  $\kernelement\in\kernfreegroup$, where $1$ has degree 0 and
  $\character(\kernelement)$ has degree $-\kernelement$. We say that
  $\freeideal$ is the sheaf of ideals of $\freealgebra$ defined by
  $\character$.  Let $\coxsheaf\coloneqq \freealgebra/\freeideal$ and
  $\freemor\colon \freealgebra\to\coxsheaf$ the projection.
  Then
 \begin{enumerate}[label=(\arabic*), ref={\it{(\arabic{*})}}]
  \item $\freeideal$ is $\gradgroup$-homogeneous, and every
    $\freegroup$-homogeneous section of $\freeideal$ is
    zero;\label{item:construction:homogeneous_ideal}
  \item $\freemor|_{\freeinvsheaf{\freeelement}}\colon 
    \freeinvsheaf{\freeelement}\to\coxsheaf_{\presentation(\freeelement)}$ is an
    isomorphism of $\OO_\Xc$-modules for all
    $\freeelement\in\freegroup$;\label{item:construction:isomorphisms}
  \item $\coxsheaf$ with a family of isomorphisms
    $\{\coxisom_{\groupelement,
      \divisor}\}_{(\groupelement,\divisor)\in\gradgroupdiv}$ induced
    by
    $\{\freemor|_{\freeinvsheaf{\freeelement}}^{-1}\}_{\freeelement\in\freegroup}$
    is a Cox sheaf of $\Xc$ of type
    $\typecox$.\label{item:construction:cox_sheaf}
  \end{enumerate}
\end{construction}

\begin{proof}
  Since $\presentation(\kernelement)=0$ for all
  $\kernelement\in\kernfreegroup$, the sheaf $\freeideal$ is
  $\gradgroup$-homogeneous and $\coxsheaf$ is $\gradgroup$-graded. To prove
  \ref{item:construction:homogeneous_ideal}, let $U$ be an open subset of
  $\Xc$ and $\coxsec\in H^0(U,\freeideal)$ a homogeneous section of degree
  $\freeelement\in\freegroup$. Let $\mathcal{V}$ be an open covering of $U$
  such that, for all $V\in\mathcal{V}$, we can write
  $\coxsec|_V=\sum_{\kernelement\in\kernfreegroup}
  (1-\character(\kernelement))\coxsec_{\kernelement}$ with
  $\coxsec_{\kernelement}\in H^0(V,\bigoplus_{\kernelement'\in\kernfreegroup}
  \freeinvsheaf{\freeelement+\kernelement'})$ for all
  $\kernelement\in\kernfreegroup$.
  Write $\coxsec_{\kernelement}=
  \sum_{\kernelement'\in\kernfreegroup}\coxsec_{\kernelement,\kernelement'}$
  with 
  $\coxsec_{\kernelement, \kernelement'}\in
  H^0(V,\freeinvsheaf{\freeelement+\kernelement'})$. Then
  \begin{equation*}
    \coxsec|_V=\sum_{\kernelement, \kernelement'\in\kernfreegroup}
    ((1-\character(\kernelement- \kernelement'))\character(\kernelement')
    \coxsec_{\kernelement, \kernelement'}-
    (1-\character(-\kernelement'))\character(\kernelement')\coxsec_{\kernelement, \kernelement'}),
  \end{equation*}
  where $\character(\kernelement')\coxsec_{\kernelement, \kernelement'}$ is
  homogeneous of degree $\freeelement$. Hence
  $\coxsec|_V=\sum_{\kernelement\in\kernfreegroup}
  (1-\character(\kernelement))\coxsec'_{\kernelement}$ for suitable
  $\coxsec'_{\kernelement}\in H^0(V,\freeinvsheaf{\freeelement})$. Since
  $\coxsec$ is homogeneous of degree $\freeelement$, we get $\coxsec|_V=0$ for
  all $V\in\mathcal{V}$. Hence $\coxsec=0$.
  
  For \ref{item:construction:isomorphisms}, note that
  $\freemor|_{\freeinvsheaf{\freeelement}}$ is injective for all
  $\freeelement\in\freegroup$ by \ref{item:construction:homogeneous_ideal}.
  Let $U$ be an open subset of $\Xc$ and $\freeelement\in\Lambda$.  We prove
  that
  $\freemor|_{H^0(U,\freeinvsheaf{\freeelement})}\colon H^0(U,\freeinvsheaf{\freeelement})\to
  H^0(U,\coxsheaf_{\presentation(\freeelement)})$ is surjective.  Let
  $\coxsec\in H^0(U,\coxsheaf_{\presentation(\freeelement)})$ and
  $\mathcal{V}$ an open covering of $U$ such that for all $V\in\mathcal{V}$
  there exists a section
  $\coxsec_V\in
  H^0(V,\bigoplus_{\kernelement\in\kernfreegroup}\freeinvsheaf{\freeelement+\kernelement})$
  with $\freemor(\coxsec_V)=s|_V$. Write
  $\coxsec_V=\sum_{\kernelement\in\kernfreegroup}\coxsec_{V,\kernelement}$
  with $\coxsec_{V,\kernelement}$ homogeneous of degree
  $\freeelement+\kernelement$. Then
  $\coxsec'_V\coloneqq 
  \coxsec_V-\sum_{\kernelement\in\kernfreegroup}(1-\chi(\kernelement))\coxsec_{V,\kernelement}$
  is homogeneous of degree $\freeelement$ and
  $\freemor(s'_V)=\coxsec|_V$. Since
  $\coxsec'_V|_{V\cap W}-\coxsec'_W|_{V\cap W}$ is a $\freegroup$-homogeneous
  section of $\freeideal$, the sections $\coxsec'_V$ glue by
  \ref{item:construction:homogeneous_ideal} to a section
  $\coxsec'\in H^0(U,\freeinvsheaf{\freeelement})$ such that
  $\freemor(\coxsec')=\coxsec$.

  Finally, we prove \ref{item:construction:cox_sheaf}.  For every
  $(\groupelement,\divisor)\in\gradgroupdiv$, choose
  $\freeelement\in\presentation^{-1}(\groupelement)$ and an
  isomorphism of $\OO_\Xc$-modules $\psi_{\freeelement,
    \divisor}\colon \freeinvsheaf{\freeelement}\to\OO_\Xc(\divisor)$. Then
  \begin{equation*}
    \left\{\coxisom_{\groupelement, \divisor}\coloneqq 
      \psi_{\freeelement,
        \divisor}\circ\freemor|_{\freeinvsheaf{\freeelement}}^{-1}\colon 
      \coxsheaf_{\groupelement}\to
      \OO_\Xc(\divisor)\right\}_{(\groupelement,\divisor)\in\gradgroupdiv}
  \end{equation*}
  is a family of isomorphisms induced by
  $\{\freemor|_{\freeinvsheaf{\freeelement}}^{-1}\}_{\freeelement\in\freegroup}$. By
  Lemma \ref{lem:isom_inv_sheaves}, for every
  $(\groupelement_1,\divisor_1),(\groupelement_2,\divisor_2)\in\gradgroupdiv$,
  there exists a nonzero constant $\compconstant\in\field$ that
  satisfies
  \begin{equation*}
    \coxisom_{\groupelement_1,
      \divisor_1}(\coxsec_1)\coxisom_{\groupelement_2,
      \divisor_2}(\coxsec_2)=\compconstant\coxisom_{\groupelement_1+\groupelement_2,
      \divisor_1+\divisor_2}(\coxsec_1\coxsec_2)
  \end{equation*}
  for all $\coxsec_i\in H^0(U,\coxsheaf_{\groupelement_i})$,
  $i\in\{1,2\}$, and all open subsets $U$ of $\Xc$.
\end{proof}

Since $\freegroup$ is free and finitely generated, the same is true for
$\kernfreegroup$.  Therefore, characters $\character$ as in Construction
\ref{construction:cox_sheaf_abstract} always exist.

\begin{prop}\label{prop:independence_of_character_abstract}
  The Cox sheaves defined in Construction
  \ref{construction:cox_sheaf_abstract} do not depend on the choice of the
  character $\character$, up to isomorphism of Cox sheaves.
\end{prop}

\begin{proof}
  Let $\freegroup$, $\kernfreegroup$, and $\freealgebra$ be as in
  Construction \ref{construction:cox_sheaf_abstract}. Let $\character$
  and $\character'$ be two characters associated with $\freealgebra$,
  and $\freeideal$ and $\freeideal'$ the sheaves of ideals of
  $\freealgebra$ defined by $\character$ and $\character'$,
  respectively.  Since $\freegroup$ is free and finitely generated,
  $\field$ is separably closed and the order of the torsion subgroup
  of $\gradgroup$ is not divisible by the characteristic of $\field$,
  the character
  $\character'\character^{-1}\colon \kernfreegroup\to\field^\times$ extends
  to a character $\alpha\colon \freegroup\to\field^\times=\globalX^\times$,
  which defines an automorphism $\morphism\colon \freealgebra\to\freealgebra$
  by sending each homogeneous element $\coxsec$ of degree
  $\freeelement$ to $\alpha(-\freeelement)\coxsec$.  Since
  $\morphism(1-\character(\kernelement))=
  1-\alpha(\kernelement)\character(\kernelement)=1-\character'(\kernelement)$
  for all $\kernelement\in\kernfreegroup$, the automorphism
  $\morphism$ maps $\freeideal$ onto $\freeideal'$. Hence it induces an
  isomorphism $\freealgebra/\freeideal\to\freealgebra/\freeideal'$ of
  $\gradgroup$-graded $\OO_\Xc$-algebras.
\end{proof}
  
\begin{prop}\label{prop:isomorphism_cox_sheaves_abstract}
  Let $\coxsheaf$ be a Cox sheaf of $\Xc$ of type $\typecox$. For every
  $\freealgebra$ as in Construction \ref{construction:cox_sheaf_abstract},
  there exists a character $\character$ associated with $\freealgebra$ and an
  isomorphism $\coxsheaf\cong\freealgebra/\freeideal$ of Cox sheaves of type
  $\typecox$, where $\freeideal$ is the sheaf of ideals of $\freealgebra$
  defined by $\character$.
\end{prop}

\begin{proof}
  Let $\kernfreegroup$, $\freegroup$, $\presentation$ and $\freealgebra$ be as
  in Construction \ref{construction:cox_sheaf_abstract}.  Let
  $\freealgebra'\coloneqq
  \bigoplus_{\freeelement\in\freegroup}\coxsheaf_{\presentation(\freeelement)}$.
  Let
  $\coxisom_{\groupelement,\divisor}\colon
  \coxsheaf_\groupelement\to\OO_\Xc(\divisor)$, for
  $(\groupelement,\divisor)\in\gradgroupdiv$, be a family of isomorphisms
  associated with $\coxsheaf$.

  Let $\basis$ be the basis of $\freegroup$ chosen in Construction
  \ref{construction:cox_sheaf_abstract}, and let $\freegroup_+$ be the monoid
  generated by $\basis$. For every $L\in\basis$, denote by
  $\divisor_\freeelement$ the Cartier divisor representing the class
  $\typecox(\presentation(\freeelement))$ in $\pic(\Xc)$ such that
  $\freeinvsheaf{\freeelement}=\OO_\Xc(\divisor_\freeelement)$. For every
  $L\in\freegroup_+$, write $\freeelement=\sum_{i=1}^r\freeelement_i$ with not
  necessarily distinct $\freeelement_i\in\basis$, define
  $\divisor_\freeelement\coloneqq \sum_{i=1}^r\divisor_{\freeelement_i}$, and let
  $\compconstant_{\freeelement}$ be the unique element of $\field^\times$ that
  satisfies
  \begin{equation*}
    \prod_{i=1}^r\coxisom_{\presentation(\freeelement_i),\divisor_{\freeelement_i}}(\coxsec_i)
    =\compconstant_{\freeelement}
    \coxisom_{\presentation(\freeelement),\divisor_{\freeelement}}(\coxsec_1\cdots\coxsec_r)
  \end{equation*}
  for all $\coxsec_i\in H^0(U,\coxsheaf_{\presentation(\freeelement_i)})$ and
  all open $U\subseteq\Xc$. For every $\freeelement\in\freegroup$, write
  $\freeelement=\freeelement^+-\freeelement^-$ with
  $\freeelement^+,\freeelement^-\in\freegroup_+$, and define
  $\compconstant_\freeelement\coloneqq \compconstant_{\freeelement^+}
  \compconstant_{\freeelement^-}^{-1}\compconstant^{-1}$, where
  $\compconstant\in\field^\times$ is the unique constant that satisfies
  $\coxisom_{\presentation(\freeelement),\divisor_{\freeelement}}(\coxsec)
  \coxisom_{\presentation(\freeelement^-),\divisor_{\freeelement^-}}(\coxsec')
  =\compconstant\coxisom_{\presentation(\freeelement^+),\divisor_{\freeelement^+}}(\coxsec\coxsec')$
  for all $\coxsec\in H^0(U,\coxsheaf_{\presentation(\freeelement)})$,
  $\coxsec'\in H^0(U,\coxsheaf_{\presentation(\freeelement^-)})$ and all open
  $U\subseteq\Xc$. The constant $\compconstant_{\freeelement}$ does not depend
  on the choice of $\freeelement^+$ and $\freeelement^-$. The morphisms
  $\morphism_\freeelement\coloneqq \compconstant_{\freeelement}
  \coxisom_{\presentation(\freeelement),\divisor_{\freeelement}}$
  induce an isomorphism of $\freegroup$-graded $\OO_\Xc$-algebras
  $\morphism\colon \freealgebra'\to\freealgebra$.

  The map $\character\colon \kernfreegroup\to\FfieldX^\times$ defined by
  $\chi(\kernelement)\coloneqq \morphism_{-\kernelement}(1)$ for all
  $\kernelement\in\kernfreegroup$, is a character associated with
  $\freealgebra$ by Lemma \ref{lem:isom_inv_sheaves}. Let $\freeideal'$ be the
  kernel of the projection $\freealgebra'\to\coxsheaf$. Then
  $\morphism(\freeideal')$ is the sheaf $\freeideal$ of ideals of
  $\freealgebra$ defined by $\character$.  Thus $\morphism$ induces an
  isomorphism $\coxsheaf\cong\freealgebra/\freeideal$.
\end{proof}
  
\begin{cor}\label{cor:existence_isom_cox_sheaves}
  There exists exactly one isomorphism class of Cox sheaves of $\Xc$ of type $\typecox$.
\end{cor}

\begin{prop}\label{prop:cox_ring_cox_sheaf}
  Every Cox ring of $\Xc$ of type $\typecox$ is isomorphic, as a Cox ring of
  type $\typecox$, to the ring of global sections of a Cox sheaf of $\Xc$ of
  type $\typecox$.
\end{prop}

\begin{proof}
  Let $\coxring$ be a Cox ring of $\Xc$ of type $\typecox$.  Let
  $\coxisom_{\groupelement,\divisor}\colon \coxring_\groupelement\to
  H^0(\Xc,\OO_\Xc(\divisor))$, for $(\groupelement,\divisor)\in\gradgroupdiv$,
  be a family of isomorphisms associated with $\coxring$ as in Proposition
  \ref{prop:equivalent_definition_cox_ring_div}.

  Let $\kernfreegroup$, $\freegroup$, $\presentation$ and $\freealgebra$ be as
  in Construction \ref{construction:cox_sheaf_abstract}.  For every
  $\freeelement\in\freegroup$, let $\divisor_\freeelement$ be the Cartier
  divisor on $\Xc$ such that
  $\freealgebra_{\freeelement}=\OO_{\Xc}(\divisor_{\freeelement})$.  Let
  $\freegroupeff$ be the subgroup of $\freegroup$ generated by the elements
  $\freeelement$ such that $\typecox(\presentation(\freeelement))$ is
  effective in $\pic(\Xc)$. We observe that
  $\kernfreegroup\subseteq\freegroupeff$.  Since $\freegroup$ is free and
  finitely generated, the same is true for $\freegroupeff$. Let $\basiseff$ be
  a basis of $\freegroupeff$, and let $\freegroup_+\subseteq\freegroupeff$ be
  the monoid generated by $\basiseff$.  For every $L\in\freegroup_+$, write
  $\freeelement=\sum_{i=1}^r\freeelement_i$ with not necessarily distinct
  $\freeelement_i\in\basiseff$, and let $\compconstant_{\freeelement}$ be the
  unique element of $\field^\times$ that satisfies
  \begin{equation*}
    \prod_{i=1}^r\coxisom_{\presentation(\freeelement_i),\divisor_{\freeelement_i}}(\coxsec_i)
    =\compconstant_{\freeelement}
    \coxisom_{\presentation(\freeelement),\divisor_{\freeelement}}(\coxsec_1\cdots\coxsec_r)
  \end{equation*}
  for all $\coxsec_i\in \coxring_{\presentation(\freeelement_i)}$ with
  $i\in\{1,\dots,r\}$. For every $\freeelement\in\freegroupeff$ such that
  $\typecox(\presentation(\freeelement))$ is effective in $\pic(\Xc)$, write
  $\freeelement=\freeelement^+-\freeelement^-$ with
  $\freeelement^+,\freeelement^-\in\freegroup_+$, and define
  $\compconstant_\freeelement\coloneqq \compconstant_{\freeelement^+}
  \compconstant_{\freeelement^-}^{-1}\compconstant^{-1}$,
  where $\compconstant\in\field^\times$ is the unique constant that satisfies
  \begin{equation*}
    \coxisom_{\presentation(\freeelement),\divisor_{\freeelement}}(\coxsec)
    \coxisom_{\presentation(\freeelement^-),\divisor_{\freeelement^-}}(\coxsec')
    =\compconstant
    \coxisom_{\presentation(\freeelement^+),\divisor_{\freeelement^+}}(\coxsec\coxsec')
  \end{equation*}
  for all $\coxsec\in \coxring_{\presentation(\freeelement)}$ and
  $\coxsec'\in\coxring_{\presentation(\freeelement^-)}$. The constant
  $\compconstant_{\freeelement}$ does not depend on the choice of
  $\freeelement^+$ and $\freeelement^-$.  The isomorphisms
  $\morphism_\freeelement
  \coloneqq \compconstant_{\freeelement}^{-1}
  \coxisom_{\presentation(\freeelement),\divisor_{\freeelement}}^{-1}$
  induce a surjective morphism of graded $\field$-algebras
  $\morphism\colon H^0(\Xc,\freealgebra)\to\coxring$.

  The map $\character\colon \kernfreegroup\to\FfieldX^\times$ defined by
  $\chi(\kernelement)\coloneqq \morphism_{-\kernelement}^{-1}(1)$ for all
  $\kernelement\in\kernfreegroup$ is a character associated with
  $\freealgebra$. Let $\freeideal$ be the sheaf of ideals of $\freealgebra$
  associated with $\character$. Then $\coxsheaf\coloneqq \freealgebra/\freeideal$ is a
  Cox sheaf of $\Xc$ of type $\typecox$,
  $H^0(\Xc,\freealgebra)/H^0(\Xc,\freeideal)\cong H^0(\Xc,\coxsheaf)$ by
  Construction
  \ref{construction:cox_sheaf_abstract}\ref{item:construction:isomorphisms},
  and $H^0(\Xc,\freeideal)$ is the kernel of $\morphism$.

  We show that the induced isomorphism
  $\morphism\colon H^0(\Xc,\freealgebra)/H^0(\Xc,\freeideal)\to\coxring$ is a
  morphism of Cox rings.  The Cox sheaf $\coxsheaf$ is endowed with the family
  of isomorphisms
  $\{\coxisom'_{\groupelement,
    \divisor}\}_{(\groupelement,\divisor)\in\gradgroupdiv}$ induced by
  $\{\freemor|_{\freeinvsheaf{\freeelement}}^{-1}\}_{\freeelement\in\freegroup}$,
  where $\freemor\colon \freealgebra\to\coxsheaf$ is the projection. For every
  $\groupelement\in\gradgroup$ such that $\typecox(\groupelement)$ is
  effective, let $\freeelement\in\freegroupeff$ be such that
  $\presentation(\freeelement)=\groupelement$. Without loss of generality, we
  can assume that
  $\coxisom_{\groupelement,
    \divisor_\freeelement}'=\freemor|_{\freeinvsheaf{\freeelement}}^{-1}$. Hence
  \begin{equation*}
    \coxisom_{\groupelement,
      \divisor_\freeelement}\circ\morphism\circ(\coxisom_{\groupelement,
      \divisor_\freeelement}')^{-1}=\compconstant_{\freeelement}^{-1}
    \identity_{H^0(\Xc,\OO_{\Xc}(\divisor_\freeelement))}.\qedhere
  \end{equation*}
\end{proof}

By definition, the degree-$\groupelement$-part of a Cox ring of $\Xc$
of type $\typecox$ is nonzero if and only if $\typecox(\groupelement)$
is an effective class in $\pic(\Xc)$.  But the
degree-$\groupelement$-part of a Cox sheaf of $\Xc$ of type $\typecox$
is always nonzero. Hence we may expect that a Cox sheaf is not
completely determined by its ring of global sections.  The following
example shows that Cox sheaves of different types, which are not
isomorphic, may have rings of global sections that are isomorphic as
$\field$-algebras.

\begin{example}
  A Cox ring of $\PP^1\times\PP^1$ of identity type is the $\ZZ^2$-graded ring
  \begin{equation*}
    k[x_0,x_1,y_0,y_1]=
    \bigoplus_{(a,b)\in\ZZ_{\geq0}\times\ZZ_{\geq0}}k[x_0,x_1,y_0,y_1]_{(a,b)} ,
  \end{equation*}
  where $x_0,x_1$ have degree $(1,0)$ and $y_0,y_1$ have degree $(0,1)$, via
  the standard identification $\pic(\PP^1\times\PP^1)\cong\ZZ^2$.
  
  Let $\lambda\colon \ZZ\to\ZZ^2$ be the morphism that sends $a$ to $(a,-a)$.  A Cox
  ring of $\PP^1\times\PP^1$ of type $\lambda$ is, by pullback (see Definition
  \ref{def:pullback}),
  \begin{equation*}
    \bigoplus_{a\in\ZZ}k[x_0,x_1,y_0,y_1]_{(a,-a)}=k,
  \end{equation*}
  which is the same, as a $\field$-algebra, as a Cox ring of $\PP^1\times\PP^1$
  of type $\{0\}\to\ZZ^2$.  A Cox sheaf of type $\typecox$ is isomorphic to
  $\bigoplus_{a\in\ZZ}\mathscr{L}^{\otimes a}$, where $\mathscr{L}$ is an
  invertible sheaf of type $(1,-1)$, while a Cox sheaf of type $\{0\}\to\ZZ^2$
  is $\OO_{\PP^1\times\PP^1}$.
\end{example}

\begin{remark}\label{rem:effective_degrees}
  The example above can be explained as follows.  Let $\gradgroupeff$ denote the
  subgroup of $\gradgroup$ generated by the elements $\groupelement$ such that
  $\typecox(\groupelement)$ is effective in $\pic(\Xc)$, and define
  $\typecoxeff\coloneqq \typecox|_{\gradgroupeff}\colon \gradgroupeff\to\pic(\Xc)$.  A Cox
  ring of $\Xc$ of type $\typecox$ also has a structure of Cox ring of $\Xc$
  of type $\typecox'$, for all $\typecox'\colon \gradgroup'\to\pic(\Xc)$ such that
  $\gradgroupeff'=\gradgroupeff$ and $\typecoxeff'=\typecoxeff$.  The
  situation in the example above cannot occur if $\typecox(\gradgroup)$ is
  generated by effective divisor classes. This is the case, for example, if
  $\typecox(\gradgroup)$ contains an ample divisor class.
\end{remark}

We recall that the morphisms of Cox sheaves of type $\typecox$ are by
definition just morphisms of graded $\OO_\Xc$-algebras, while the
morphisms of Cox rings are morphisms of graded $\field$-algebras that
are compatible with the map $\divi$.  The following example shows that
a $\field$-algebra morphism respecting the grading is not necessarily
a morphism of Cox rings of type $\typecox$.

\begin{example}
  For $\Xc\coloneqq \PP^1_\field$, the Picard group $\pic(X)$ is free of rank $1$, and
  $\coxring \cong \field[T_0,T_1]$ is a Cox ring of $\Xc$ of type
  $\identity_{\pic(\Xc)}$, where the $\pic(\Xc)$-grading is the usual
  $\ZZ$-grading by the total degree, identifying effective divisor classes
  with nonnegative integers. Mapping $T_0$ and $T_1$ to arbitrary linearly
  independent linear polynomials in $T_0,T_1$ defines an automorphism of
  $\field[T_0,T_1]$ respecting the grading, but every Cox ring automorphism of
  $\coxring$ in the sense of Definition~\ref{def:cox_ring} is multiplication
  by a scalar in $\field^\times$.
\end{example}

We observe that every morphism $\coxring\to\coxring'$ of Cox rings of type
$\typecox$ in the sense of Definition \ref{def:cox_ring} (cf.~Remark
\ref{rem:equivalent_characterization_cox_ring}) is an isomorphism, because it
restricts to an isomorphism
$\coxring_{\groupelement}\to\coxring'_{\groupelement}$ for every
$\groupelement\in\gradgroup$, as it is compatible with the map $\divi$.

\begin{prop}\label{prop:morphism_cox_sheaves}
Every morphism of Cox sheaves of $\Xc$ of type $\typecox$ is an isomorphism.
\end{prop}

\begin{proof}
  Let $\morphism\colon \coxsheaf\to\coxsheaf'$ be a morphism of Cox sheaves
  of $\Xc$ of type $\typecox$. Let
  $\{\coxisom_{\groupelement,\divisor}\}_{(\groupelement,\divisor)\in\gradgroupdiv}$
  and
  $\{\coxisom'_{\groupelement,\divisor}\}_{(\groupelement,\divisor)\in\gradgroupdiv}$
  be families of isomorphisms associated to $\coxsheaf$ and
  $\coxsheaf'$, respectively. For every $\groupelement\in\gradgroup$,
  fix a Cartier divisor $\divisor_{\groupelement}$ such that
  $(\groupelement,\divisor_\groupelement)\in\gradgroupdiv$, and let
  $\compconstant_{\groupelement}\in\globalX$ such that
  \begin{equation*}
    \coxisom'_{\groupelement,\divisor_{\groupelement}}
    \circ\morphism|_{\coxsheaf_{\groupelement}}\circ
    \coxisom_{\groupelement,\divisor_{\groupelement}}^{-1}
    =\compconstant_{m}\identity_{\OO_{\Xc}(\divisor_{\groupelement})}
  \end{equation*}
  (cf.~Lemma \ref{lem:isom_inv_sheaves}).  Let $\opensubset$ be an
  open subset of $\Xc$ that trivializes all elements in
  $\typecox(\gradgroup)$. Fix $\groupelement\in\gradgroup$. For every
  $\coxsec^+\in\coxsheaf_{\groupelement}(\opensubset)$ and
  $\coxsec^-\in\coxsheaf_{-\groupelement}(\opensubset)$,
  \begin{align*}
    &\beta^{-1}\compconstant_0\coxisom_{\groupelement,\divisor_{\groupelement}}(\coxsec^+)
      \coxisom_{-\groupelement,\divisor_{-\groupelement}}(\coxsec^-)=
      \coxisom'_{0,\divisor_0}(\morphism|_{\coxsheaf_0}(\coxsec^+\coxsec^-))\\
    &=\coxisom'_{0,\divisor_0}(\morphism|_{\coxsheaf_m}(\coxsec^+)
      \morphism|_{\coxsheaf_{-m}}(\coxsec^-))=
      (\beta')^{-1}\compconstant_{\groupelement}\compconstant_{-\groupelement}
      \coxisom_{\groupelement,\divisor_{\groupelement}}(\coxsec^+)
      \coxisom_{-\groupelement,\divisor_{-\groupelement}}(\coxsec^-),
  \end{align*}
  where $\beta, \beta'\in\field^\times$ are the unique constants such that 
  \begin{align*}
    &\coxisom_{\groupelement,\divisor_{\groupelement}}(\coxsec^+)
      \coxisom_{-\groupelement,\divisor_{-\groupelement}}(\coxsec^-)
      =\beta\coxisom_{0,\divisor_{0}}(\coxsec^+\coxsec^-),\\
    &\coxisom'_{\groupelement,\divisor_{\groupelement}}(\coxsec^+)
      \coxisom'_{-\groupelement,\divisor_{-\groupelement}}(\coxsec^-)
      =\beta'\coxisom'_{0,\divisor_{0}}(\coxsec^+\coxsec^-),
  \end{align*}
  respectively. Then
  $\compconstant_{\groupelement}\compconstant_{-\groupelement}
  =\beta'\beta^{-1}\compconstant_0$,
  as $\beta$ and $\beta'$ do not depend on the choice of $\coxsec^+$ and
  $\coxsec^-$.
  Since $\morphism$ is a morphism of $\OO_\Xc$-algebras and both $\coxsheaf_0$
  and $\coxsheaf_0'$ are isomorphic to $\OO_{\Xc}$, we have $\morphism(1)=1$,
  and $\morphism|_{\coxsheaf_0}\colon \coxsheaf_0\to\coxsheaf_0'$ is an isomorphism
  of $\OO_{\Xc}$-modules. Hence
  $\compconstant_0\in\globalX^\times=\field^\times$ by Lemma
  \ref{lem:isom_inv_sheaves}.
\end{proof}

As a consequence of Proposition \ref{prop:morphism_cox_sheaves} and
Lemma \ref{lem:isom_inv_sheaves}, every morphism of Cox sheaves of
$\Xc$ of type $\typecox$ induces a morphism of Cox rings of $\Xc$ of
type $\typecox$ between the rings of global sections in the sense of
Definition \ref{def:cox_ring}. Hence, for every $\typecox$, there
exists exactly one isomorphism class of Cox rings of $\Xc$ of type
$\typecox$ by Corollary \ref{cor:existence_isom_cox_sheaves}.

\begin{prop}\label{prop:morphism_cox_ring_cox_sheaf}
  Let $\coxsheaf$ and $\coxsheaf'$ be two Cox sheaves of $\Xc$ of type
  $\typecox$.  If $\gradgroup=\gradgroupeff$, then every morphism
  $\coxsheaf(\Xc)\to\coxsheaf'(\Xc)$ of Cox rings is induced by a
  unique morphism $\coxsheaf\to\coxsheaf'$ of Cox sheaves.
\end{prop}

\begin{proof}
  Let
  $\{\coxisom_{\groupelement,\divisor}\}_{(\groupelement,\divisor)\in\gradgroupdiv}$
  and
  $\{\coxisom'_{\groupelement,\divisor}\}_{(\groupelement,\divisor)\in\gradgroupdiv}$
  be families of isomorphisms associated with $\coxsheaf$ and
  $\coxsheaf'$, respectively. For every $(\groupelement_1,\divisor_1),
  (\groupelement_2,\divisor_2)\in\gradgroupdiv$, let
  $\beta_{\groupelement_1,\groupelement_2;\divisor_1,\divisor_2}$ and
  $\beta'_{\groupelement_1,\groupelement_2;\divisor_1,\divisor_2}$ be the
  unique constants such that
  \begin{align*}
    \coxisom_{\groupelement_1, \divisor_1}(\coxsec_1)
    \coxisom_{\groupelement_2, \divisor_2}(\coxsec_2)
    &=\beta_{\groupelement_1,\groupelement_2;\divisor_1,\divisor_2}
      \coxisom_{\groupelement_1+\groupelement_2, \divisor_1+\divisor_2}(\coxsec_1\coxsec_2),\\
    \coxisom'_{\groupelement_1,\divisor_1}(\coxsec'_1)
    \coxisom'_{\groupelement_2,\divisor_2}(\coxsec'_2)
    &=\beta'_{\groupelement_1,\groupelement_2;\divisor_1,\divisor_2}
      \coxisom'_{\groupelement_1+\groupelement_2,\divisor_1+\divisor_2}(\coxsec'_1\coxsec'_2)
  \end{align*}
  for all $\coxsec_i\in \coxsheaf_{\groupelement_i}(U)$ and
  $\coxsec'_i\in \coxsheaf'_{\groupelement_i}(U)$, $i\in\{1,2\}$, and all open
  subsets $U$ of $\Xc$.

  Let $\morphism\colon \coxsheaf(\Xc)\to\coxsheaf'(\Xc)$ be a morphism of
  Cox rings of type $\typecox$.
  By definition of morphism of Cox rings of $\Xc$ of type $\typecox$
  (cf.~Remark \ref{rem:equivalent_characterization_cox_ring}), for every
  $(\groupelement,\divisor)\in\gradgroupdiv$ such that $\typecox(\groupelement)$ is
  effective, we have
  $\morphism|_{\coxsheaf(\Xc)_{\groupelement}}
  =\compconstant_{\groupelement,\divisor}(\coxisom'_{\groupelement,\divisor})^{-1}
  \circ\coxisom_{\groupelement,\divisor}\colon \coxsheaf(\Xc)_{\groupelement}
  \to\coxsheaf'(\Xc)_{\groupelement}$ for a constant
  $\compconstant_{\groupelement,\divisor}\in\field^\times$.
  Moreover,
  \begin{equation*}
    \compconstant_{\groupelement_1+\groupelement_2,\divisor_1+\divisor_2}
    =\compconstant_{\groupelement_1,\divisor_1}\compconstant_{\groupelement_2,\divisor_2}
    \beta_{\groupelement_1,\groupelement_2;\divisor_1,\divisor_2}
    (\beta'_{\groupelement_1,\groupelement_2;\divisor_1,\divisor_2})^{-1}
  \end{equation*}
  for all $(\groupelement_1,\divisor_1),(\groupelement_2,\divisor_2)\in\gradgroupdiv$ as above, as
  $\morphism$ is compatible with the multiplication in
  $\coxsheaf(\Xc)$ and in $\coxsheaf'(\Xc)$.
  
  For every $\groupelement\in\gradgroup$, write
  $\groupelement=\groupelement_1-\groupelement_2$ with
  $\groupelement_i\in\gradgroup$ such that $\typecox(\groupelement_i)$
  is effective, choose Cartier divisors $\divisor_i$
  such that $(\groupelement_i,\divisor_i)\in\gradgroupdiv$,
  $i\in\{1,2\}$, and let $\divisor\coloneqq \divisor_1-\divisor_2$.  The
  isomorphisms
  \begin{equation*}
    \compconstant_{\groupelement_1,\divisor_1}
    \compconstant_{\groupelement_2,\divisor_2}^{-1}
    \beta_{\groupelement,\groupelement_2;\divisor,\divisor_2}^{-1}
    \beta'_{\groupelement,\groupelement_2;\divisor,\divisor_2}
    {\coxisom'_{\groupelement, \divisor}}^{-1}\circ\coxisom_{\groupelement,
      \divisor}
    \colon  \coxsheaf_\groupelement \to \coxsheaf'_\groupelement
  \end{equation*}
  define a morphism of Cox sheaves $\coxsheaf\to\coxsheaf'$ that
  induces $\morphism$ at the level of global sections.
\end{proof}

\begin{prop}\label{prop:cox_sheaf_aut}
  Let $\coxsheaf$ be a Cox sheaf of $\Xc$ of type $\typecox$. For
  every
  $\dualgroupelement\in\gradgroupdual(\field)=\Hom(\gradgroup,\field^\times)$,
  let $\automorphism_{\dualgroupelement}\colon \coxsheaf\to\coxsheaf$ be the
  map defined as scalar multiplication by
  $\dualgroupelement(\groupelement)$ on $\coxsheaf_\groupelement$ for
  all $\groupelement\in\gradgroup$. Then
  $\dualgroupelement\mapsto\automorphism_{\dualgroupelement}$ defines
  an isomorphism between $\gradgroupdual(\field)$ and the group of Cox
  sheaf automorphisms of $\coxsheaf$.
\end{prop}

\begin{proof}
  By Lemma \ref{lem:isom_inv_sheaves}, a Cox sheaf automorphism
  $\automorphism$ of $\coxsheaf$ must be scalar multiplication by some
  $\dualgroupelement_\groupelement \in \field^\times$ on each
  homogeneous part $\coxsheaf_\groupelement$.  Moreover,
  $\dualgroupelement_\groupelement
  \dualgroupelement_{\groupelement'}=\dualgroupelement_{\groupelement+\groupelement'}$
  for all $\groupelement,\groupelement'\in\gradgroup$ as
  $\automorphism$ is compatible with the multiplication in
  $\coxsheaf$.  So, $\groupelement \mapsto
  \dualgroupelement_\groupelement$ defines a group homomorphism
  $\gradgroup \to \field^\times$, and hence an element
  $\dualgroupelement \in \gradgroupdual(\field)$ such that
  $\automorphism = \automorphism_{\dualgroupelement}$.
\end{proof}

As a consequence of Propositions \ref{prop:morphism_cox_sheaves},
\ref{prop:morphism_cox_ring_cox_sheaf} and \ref{prop:cox_sheaf_aut}, the group
of Cox ring automorphisms of a Cox ring of $\Xc$ of type $\typecox$ is
isomorphic to $\gradgroupeffdual(\field)$.

\begin{prop}\label{prop:cox_sheaves_torsors}
  If $\torsormor\colon \Yc\to\Xc$ is an $\Xc$-torsor under $\gradgroupdual$ of type
  $\typecox$, then $\torsormor_*\OO_{\Yc}$ is a Cox sheaf of $\Xc$ of type
  $\typecox$. Conversely, if $\coxsheaf$ is a Cox sheaf of $\Xc$ of type
  $\typecox$, then the relative spectrum $\spec_\Xc\coxsheaf\to\Xc$ is an
  $\Xc$-torsor under $\gradgroupdual$ of type $\typecox$.
\end{prop}

\begin{proof}
  Let $\torsormor\colon \Yc\to\Xc$ be an $\Xc$-torsor under $\gradgroupdual$ of type
  $\typecox$.  Let $\{U_i\}_i$ be an open covering of $\Xc$ that trivializes
  $\torsormor$, and let $(\cocycleY_{i,j})_{i,j}$ with
  \begin{equation*}
    \cocycleY_{i,j}\in \Hom(\gradgroup,\OO_{\Xc}(U_i\cap U_j)^\times)
    \cong\gradgroupdual(U_i\cap U_j)
  \end{equation*}
  be a cocycle representing
  $[\Yc]\in H^1_{\text{\it{\'et}}}(\Xc,\gradgroupdual)$.  By definition of the
  $\type$ map (cf.~Section~\ref{subsection:generalized_cox_rings}), for every
  $\groupelement\in\gradgroup$, there exists a Cartier divisor
  $\divisor_\groupelement=\{(U_i,f_{i,\groupelement})\}_i$ such that
  $[\divisor_\groupelement]=\typecox(\groupelement)$ in $\pic(\Xc)$ and
  $\cocycleY_{i,j}(\groupelement)=f_{i,\groupelement}f_{j,\groupelement}^{-1}$.
  The isomorphisms of $\OO_{U_i}$-algebras
  \begin{equation*}
    \torsormor_*\OO_{\Yc}|_{U_i}\cong\OO_{U_i}[\gradgroup]\cong
    \bigoplus_{\groupelement\in\gradgroup}f^{-1}_{i,\groupelement}\OO_{U_i}
  \end{equation*} 
  glue over the cocycle $(\cocycleY_{i,j})_{i,j}$ to give an isomorphism of
  $\OO_{\Xc}$-modules
  \begin{equation*}
    \coxisom\colon \torsormor_*\OO_{\Yc}\to\bigoplus_{\groupelement\in\gradgroup}
    \OO_\Xc(\divisor_\groupelement).
  \end{equation*}
  Hence $\torsormor_*\OO_{\Yc}$ is an $\gradgroup$-graded $\OO_\Xc$-algebra.
  For all $\groupelement\in\gradgroup$, let
  $\coxisom_{\groupelement,\divisor_\groupelement}$ be the isomorphism
  \begin{equation*}
    \coxisom|_{(\torsormor_*\OO_{\Yc})_\groupelement}\colon (\torsormor_*\OO_{\Yc})_\groupelement
    \to\OO_\Xc(\divisor_\groupelement),
  \end{equation*}
  where
  $(\torsormor_*\OO_{\Yc})_\groupelement$ is the degree-$\groupelement$-part of
  $\torsormor_*\OO_{\Yc}$.
  Let $U$ be an open subset of $\Xc$, and
  $\coxsec_1,\coxsec_2\in\torsormor_*\OO_{\Yc}(U)$ homogeneous of degree
  $\groupelement_1, \groupelement_2$, respectively. We observe that, for all
  $i$,
  \begin{equation*}
    \coxisom(\coxsec_1\coxsec_2)|_{U\cap U_i}
    =f_{i,\groupelement_1+\groupelement_2}^{-1}
    f_{i,\groupelement_1}\coxisom(\coxsec_1)|_{U\cap U_i}
    f_{i,\groupelement_2}\coxisom(\coxsec_2)|_{U\cap U_i},
  \end{equation*}
  where the product on the right is computed in $\FfieldX$. Since
  $f_{i,\groupelement_1+\groupelement_2}^{-1}f_{i,\groupelement_1}f_{i,\groupelement_2}$
  belongs to $\OO_\Xc(U_i)^\times$ for all $i$, and
  $f_{i,\groupelement_1+\groupelement_2}^{-1}f_{i,\groupelement_1}f_{i,\groupelement_2}
  =f_{j,\groupelement_1+\groupelement_2}^{-1}f_{j,\groupelement_1}f_{j,\groupelement_2}$
  in $\FfieldX^\times$ for all $i,j$ by restricting
  $\coxisom$ and $\coxsec_1, \coxsec_2$ to $U\cap U_i\cap U_j$, the element
  $\compconstant\coloneqq f_{i,\groupelement_1+\groupelement_2}f_{i,\groupelement_1}^{-1}
  f_{i,\groupelement_2}^{-1}$ belongs to $\globalX^\times=\field^\times$, and
  \begin{equation*}
    \coxisom_{\groupelement_1,\divisor_{\groupelement_1}}(\coxsec_1)
    \coxisom_{\groupelement_2,\divisor_{\groupelement_2}}(\coxsec_2)=
    \compconstant
    \coxisom_{\groupelement_1+\groupelement_2,\divisor_{\groupelement_1+\groupelement_2}}
    (\coxsec_1\coxsec_2),
  \end{equation*}
  for all $\coxsec_1,\coxsec_2\in\torsormor_*\OO_{\Yc}(U)$ homogeneous of
  degree $\groupelement_1, \groupelement_2$, respectively.

  Conversely, let $\coxsheaf$ be a Cox sheaf of $\Xc$ of type $\typecox$. The
  morphism $\torsormor\colon \spec_\Xc\coxsheaf\to\Xc$ induced by
  $\OO_\Xc\subseteq\coxsheaf$ is surjective.  For every open subset $U$ of
  $\Xc$ that trivializes all elements of $\typecox(\gradgroup)$, there are
  isomorphisms
  \begin{equation*}
    \coxsheaf|_{U}\cong\OO_U[\gradgroup]\cong\bigoplus_{\groupelement\in\gradgroup}\OO_{U}.
  \end{equation*}
  Therefore, $\coxsheaf$ is locally free as an $\OO_\Xc$-module and locally
  finitely generated as an $\OO_\Xc$-algebra, and the morphism $\torsormor$ is
  flat and of finite type. Moreover, $\spec_\Xc\coxsheaf$ is locally
  isomorphic to $\gradgroupdual_{\Xc}\cong\spec_{\Xc}\OO_\Xc[\gradgroup]$ with
  the natural action of $\gradgroupdual_{\Xc}$ on itself. Hence $\torsormor$
  is an $\Xc$-torsor under $\gradgroupdual$.

  Let $\{U_i\}_i$ be an open covering that trivializes all elements of
  $\typecox(\gradgroup)$. Then it trivializes $\torsormor$. Let
  $(\cocycleY_{i,j})_{i,j}$ with
  \begin{equation*}
    \cocycleY_{i,j}\in\Hom(\gradgroup,\OO_\Xc(U_i\cap U_j)^\times)
  \end{equation*}
  be a cocycle representing the class of $\spec_\Xc\coxsheaf$ in
  $H^1_{\text{\it{\'et}}}(\Xc,\gradgroupdualX)$. By definition of the $\type$
  map, the cocycle $(\cocycleY_{i,j}(\groupelement))_{i,j}$ represents the
  isomorphism class of $\coxsheaf_{\groupelement}$ in $\pic(\Xc)$ for all
  $\groupelement\in\gradgroup$. Since $\coxsheaf$ is a Cox sheaf of type
  $\typecox$, for every $\groupelement\in\gradgroup$, there is an isomorphism
  $\coxsheaf_\groupelement\cong\OO_X(\divisor_m)$, where
  $\divisor_\groupelement$ is a Cartier divisor such that
  $[\divisor_\groupelement]=\typecox(\groupelement)$. Therefore, the cocycle
  $(\cocycleY_{i,j}(m))_{i,j}$ represents $\typecox(\groupelement)$ in
  $\Pic(\Xc)$, and $\type([\spec_\Xc\coxsheaf])=\typecox$.
\end{proof}

\section{Over nonclosed fields}\label{section:over_nonclosed_fields}

Let $\field$ be a field. We fix a separable closure $\fieldbar$ of $\field$,
with Galois group $\galois\coloneqq \gal(\fieldbar/\field)$, and every algebraic
extension of $\field$ mentioned later is contained in $\fieldbar$.  In this
section, $\Xc$ always denotes a $\field$-variety such that
$\Xcbar\coloneqq \Xc\times_{\spec\field}\spec\fieldbar$ has only constant invertible
regular functions (i.e., $\globalXbar^\times=\fieldbar^\times$, where
$\globalXbar\coloneqq  H^0(\Xcbar,\OO_{\Xcbar})$).

The action of $\galois$ on $\fieldbar$ induces an action on
$A\otimes_\field \fieldbar$ (with $\galoiselement \in \galois$ acting
via $\identity_A \otimes \galoiselement$) for every $\field$-algebra
$A$, and similarly on $\OO_{\Xcbar} \cong \OO_\Xc
\otimes_\field\fieldbar$. The action of an element
$\galoiselement\in\galois$ on $\Xcbar$ is the one induced by the
action of $\galoiselement^{-1}$ on $\OO_{\Xcbar}$.  For every
$\galoiselement\in \galois$, we denote by $\galoiselement\divisor$ the
natural Galois action on a divisor $\divisor\in\CaDiv(\Xcbar)$, by
$\galoiselement(f)$ the natural Galois action on an element
$f\in\fieldbar(\Xc)$. All these actions are \emph{continuous} (with
respect to the Krull topology on $\galois$ and the discrete topology
on the other objects).  We will denote by $\galoiselement*\coxsec$ an
action of $\galoiselement \in \galois$ on a section $\coxsec$ of a Cox
sheaf (or an element $\coxsec$ of a Cox ring) of $\Xcbar$.

From here on, $\gradgroup$ denotes a $\galois$-module that is finitely
generated as an abelian group, and whose torsion is coprime to the
characteristic of $\field$.  Let $\typecox\colon \gradgroup\to\pic(\Xcbar)$ be
a homomorphism of $\galois$-modules.  Let
$\gradgroupdual\coloneqq \spec\fieldbar[\gradgroup]^\galois$ be the quasitorus
dual to $\gradgroup$ under the antiequivalence of categories recalled in
Section~\ref{subsection:generalized_cox_rings}. Here,
$\fieldbar[\gradgroup]^\galois$ denotes the subring of $\galois$-invariant
elements of $\fieldbar[\gradgroup]$.  As in
Remark~\ref{rem:effective_degrees}, let $\gradgroupeff$ be the subgroup of
$\gradgroup$ generated by the elements $\groupelement$ such that
$\typecox(\groupelement)$ is effective in $\pic(\Xcbar)$.

\begin{defin}\label{def:galois_action}
  A continuous $\galois$-action on a Cox ring $\coxring$ of $\Xcbar$
  of type $\typecox$ is called \emph{\compatible} if
  $\divi(\galoiselement*\coxsec)=\galoiselement\divi(\coxsec)$ for all
  nonzero homogeneous $\coxsec\in\coxring$ and all
  $\galoiselement\in\galois$.  A \emph{$\galois$-equivariant Cox ring
    of $\Xcbar$ of type $\typecox$} is a Cox ring of $\Xcbar$ of type
  $\typecox$ with a {\compatible} $\galois$-action.

  A continuous $\galois$-action on a Cox sheaf $\coxsheaf$ of $\Xcbar$
  of type $\typecox$ is called \emph{\compatible} if, given an
  associated family of isomorphisms
  $\{\coxisom_{\groupelement,\divisor}\}_{(\groupelement,\divisor)\in\gradgroupdiv}$,
  for every $\galoiselement\in \galois$ and
  $(\groupelement,\divisor)\in\gradgroupdiv$, the automorphism of
  $\coxsheaf$ defined by the action of $\galoiselement$ restricts to
  an isomorphism $\coxsheaf_{\groupelement}\to
  \coxsheaf_{\galoiselement\groupelement}$ such that
  $\galoiselement^{-1}\circ\coxisom_{\galoiselement\groupelement,\galoiselement\divisor}\circ
  \galoiselement\circ\coxisom_{\groupelement,\divisor}^{-1}$ is an
  automorphism of $\OO_{\Xcbar}(\divisor)$.  A
  \emph{$\galois$-equivariant Cox sheaf of $\Xcbar$ of type
    $\typecox$} is a Cox sheaf of $\Xcbar$ of type $\typecox$ with a
  {\compatible} $\galois$-action.
  
  A \emph{morphism of $\galois$-equivariant Cox sheaves (rings) of
    $\Xcbar$ of type $\typecox$} is a $\galois$-equi\-variant morphism
  of Cox sheaves (respectively, rings) of $\Xcbar$ of type $\typecox$.
\end{defin}

\begin{remark}\label{rem:equivalent_galois_action}
  Using the equivalent definition of Cox ring provided by Proposition
  \ref{prop:equivalent_definition_cox_ring_div}, a continuous $\galois$-action
  on a Cox ring $\coxring$ of $\Xcbar$ of type $\typecox$ is
  \emph{\compatible} if and only if, given an associated family of
  isomorphisms
  $\{\coxisom_{\groupelement,\divisor}\}_{(\groupelement,\divisor)\in\gradgroupdiv}$,
  for every $\galoiselement\in \galois$ and
  $(\groupelement,\divisor)\in\gradgroupdiv$, the automorphism of $\coxring$
  defined by the action of $\galoiselement$ restricts to an isomorphism
  $\coxring_{\groupelement}\to \coxring_{\galoiselement\groupelement}$ such
  that
  $\galoiselement^{-1}\circ
  \coxisom_{\galoiselement\groupelement,\galoiselement\divisor}\circ
  \galoiselement\circ\coxisom_{\groupelement,\divisor}^{-1}=
  \compconstant\identity_{H^0(\Xcbar,\OO_{\Xcbar}(\divisor))}$
  with $\compconstant\in\fieldbar^\times$.  We observe that a natural
  $\galois$-action on a Cox sheaf $\coxsheaf$ of $\Xcbar$ of type $\typecox$
  induces a natural $\galois$-action on the Cox ring $\coxsheaf(\Xcbar)$ of
  $\Xcbar$ of type $\typecox$.
\end{remark}

\begin{defin}\label{def:cox_sheaf_nonclosed}
  A \emph{Cox ring of $\Xc$ of type $\typecox$} is a $\field$-algebra
  $\coxring$ together with a structure of Cox ring of $\Xcbar$ of type
  $\typecox$ on $\coxringbar\coloneqq \coxring\otimes_\field\fieldbar$ such that
  $\coxringbar$ with the induced action of $\galois$ is a
  $\galois$-equivariant Cox ring of $\Xcbar$.  A \emph{Cox sheaf of $\Xc$ of
    type $\typecox$} is a sheaf $\coxsheaf$ of $\OO_{\Xc}$-algebras together
  with a structure of Cox sheaf of $\Xcbar$ of type $\typecox$ on
  $\coxsheafbar\coloneqq \coxsheaf\otimes_\field\fieldbar$ such that $\coxsheafbar$
  with the induced action of $\galois$ is a $\galois$-equivariant Cox sheaf of
  $\Xcbar$.
  
  A \emph{morphism of Cox rings of $\Xc$ of type $\typecox$} is a morphism of
  $\field$-algebras $\morphism\colon \coxring\to\coxring'$ such that
  $\morphism\otimes\identity_{\fieldbar}\colon \coxringbar\to\coxringbar'$ is a
  morphism of Cox rings of $\Xcbar$ of type $\typecox$.  A \emph{morphism of
    Cox sheaves of $\Xc$ of type $\typecox$} is a morphism of
  $\OO_{\Xc}$-algebras $\morphism\colon \coxsheaf\to\coxsheaf'$ such that
  $\morphism\otimes\identity_{\fieldbar}\colon \coxsheafbar\to\coxsheafbar'$ is a
  morphism of Cox sheaves of $\Xcbar$ of type $\typecox$.
\end{defin}

\begin{prop}\label{prop:galois_descent_cox_rings_sheaves}
  The covariant functors
  \begin{align*}
    \big\{\text{Cox rings of $\Xc$ of type $\typecox$}\big\}
    & \longrightarrow \left\{
      \begin{gathered}
        \text{$\galois$-equivariant Cox rings}\\
        \text{of $\Xcbar$ of type $\typecox$}
      \end{gathered}
    \right\}\\
    \coxring\quad
    & \longmapsto\quad \coxringbar
  \end{align*}
  and
  \begin{align*}
    \big\{\text{Cox sheaves of $\Xc$ of type $\typecox$}\big\}
    & \longrightarrow\left\{
      \begin{gathered}
        \text{$\galois$-equivariant Cox sheaves}\\
        \text{of $\Xcbar$ of type $\typecox$}
      \end{gathered}
    \right\}\\
    \coxsheaf\quad
    & \longmapsto\quad \coxsheafbar
  \end{align*}  
  are equivalences of categories, with inverse functor $H^0(\galois,\cdot)$.
\end{prop}

\begin{proof}
  Let $\coxring$ be a $\galois$-equivariant Cox ring of $\Xcbar$ of type
  $\typecox$, and $\coxring^\galois\coloneqq H^0(\galois,\coxring)$ its subring of
  $\galois$-invariant elements.  Since the action of $\galois$ on $\coxring$
  is continuous, there is an isomorphism
  $\coxring^\galois\otimes_\field\fieldbar\cong \coxring$ by \cite[Proposition
  16.15]{Milne}.  Similarly, if $\coxsheaf$ is a $\galois$-equivariant Cox
  sheaf of $\Xcbar$ of type $\typecox$, then the sheaf $\coxsheaf^\galois$
  defined by $\coxsheaf^\galois(U)\coloneqq H^0(\galois,\coxsheaf(U_{\fieldbar}))$ for
  all open subsets $U$ of $\Xc$ is a Cox sheaf of $\Xc$ of type $\typecox$.
  Moreover, if $\psi\colon \coxsheaf\to\coxsheaf'$ is a morphism of
  $\galois$-equivariant Cox sheaves of $\Xcbar$ of type $\typecox$, then
  $\psi(\coxsheaf(\opensubsetbar)^\galois)
  \subseteq\coxsheaf'(\opensubsetbar)^\galois$ for every open subset $U$ of
  $\Xc$. Hence $\psi$ restricts to a unique morphism
  $\psi^\galois\colon \coxsheaf^\galois\to{\coxsheaf'}^\galois$ of Cox sheaves of
  $\Xc$ of type $\typecox$ such that
  $\psi=\psi^\galois\otimes\identity_{\fieldbar}$ under the identifications
  $\coxsheaf\cong(\coxsheaf^\galois)_{\fieldbar}$ and
  $\coxsheaf'\cong({\coxsheaf'}^\galois)_{\fieldbar}$.
\end{proof}

\begin{prop}\label{prop:equiv_cox_ring_sheaf}
  Let $\coxsheaf$ be a Cox sheaf of $\Xcbar$ of type $\typecox$. If
  $\gradgroup=\gradgroupeff$, then every {\compatible}
  $\galois$-action on $\coxsheaf(\Xcbar)$ is induced by a
  {\compatible} $\galois$-action on $\coxsheaf$.
\end{prop}

\begin{proof}
  Assume that $\coxsheaf(\Xcbar)$ is endowed with a {\compatible}
  $\galois$-action.  Let
  $\coxisom_{\groupelement,\divisor}\colon \coxsheaf_{\groupelement}\to\OO_{\Xcbar}(\divisor)$,
  for $(\groupelement,\divisor)\in\gradgroupdiv$, be a family of isomorphisms
  associated with $\coxsheaf$. For every $\galoiselement\in\galois$ and
  $(\groupelement,\divisor)\in\gradgroupdiv$ such that
  $\typecox(\groupelement)$ is effective in $\pic(\Xcbar)$, let
  $\compconstant_{\galoiselement,\groupelement,\divisor}\in\fieldbar^\times$
  be the unique constant such that
  \begin{equation*}
    \coxisom_{\galoiselement\groupelement,\galoiselement\divisor}\circ
    \galoiselement=\compconstant_{\galoiselement,\groupelement,\divisor}
    \galoiselement\circ\coxisom_{\groupelement,\divisor}.
  \end{equation*}
  For an arbitrary $(\groupelement,\divisor)\in\gradgroupdiv$, let
  $(\groupelement^+,\divisor^+),(\groupelement^-,\divisor^-)\in\gradgroupdiv$
  such that $\groupelement=\groupelement^+-\groupelement^-$,
  $\divisor=\divisor^+-\divisor^-$, and $\typecox(\groupelement^+)$ and
  $\typecox(\groupelement^-)$ are both effective in $\pic(\Xcbar)$. We define
  \begin{equation*}
    \compconstant_{\galoiselement,\groupelement,\divisor}
    \coloneqq \compconstant_{\galoiselement,\groupelement^+,\divisor^+}
    \compconstant_{\galoiselement,\groupelement^-,\divisor^-}^{-1}
    \beta\galoiselement(\gamma^{-1}),
  \end{equation*}
  where $\beta,\gamma\in\fieldbar^\times$ are the unique constants such that
  \begin{align*}
    \coxisom_{\groupelement,\divisor}(\coxsec)
    \coxisom_{\groupelement^-,\divisor^-}(\coxsec')
    &=\gamma\coxisom_{\groupelement^+,\divisor^+}(\coxsec\coxsec'),
    \\
    \coxisom_{\galoiselement\groupelement,\galoiselement\divisor}(\galoiselement*\coxsec)
    \coxisom_{\galoiselement\groupelement^-,\galoiselement\divisor^-}(\galoiselement*\coxsec')
    &=\beta\coxisom_{\galoiselement\groupelement^+,\galoiselement\divisor^+}
    (\galoiselement*(\coxsec\coxsec'))
  \end{align*}
  for all $\coxsec\in\coxsheaf(\Xcbar)_{\groupelement}$ and
  $\coxsec'\in\coxsheaf(\Xcbar)_{\groupelement^-}$. The constant
  $\compconstant_{\galoiselement,\groupelement,\divisor}$ does not depend on
  the choice of $(\groupelement^+,\divisor^+)$ and
  $(\groupelement^-,\divisor^-)$.

  For every homogeneous section $\coxsec$ of $\coxsheaf$ and
  $\galoiselement\in\galois$, let
  $\galoiselement*\coxsec \coloneqq
  \compconstant_{\galoiselement,\groupelement,\divisor}
  (\coxisom_{\galoiselement\groupelement,\galoiselement\divisor}^{-1}\circ
  \galoiselement\circ\coxisom_{\groupelement,\divisor})(\coxsec)$ for all
  $(\groupelement,\divisor)\in\gradgroupdiv$ such that $\coxsec$ has degree
  $\groupelement$. This definition does not depend on the choice of the
  representative $\divisor$ for $\typecox(\groupelement)$ and induces a
  {\compatible} action of $\galois$ on $\coxsheaf$ in the sense of Definition
  \ref{def:galois_action}.
\end{proof}

\begin{cor}\label{cor:cox_ring_cox_sheaf_nonclosed}
  If $\coxsheaf$ is a Cox sheaf of $\Xc$ of type $\typecox$, then
  $\coxsheaf(\Xc)$ is a Cox ring of $\Xc$ of type $\typecox$.  Moreover, if
  $\gradgroup=\gradgroupeff$, the covariant functor
  \begin{align*}
    \big\{\text{Cox sheaves of $\Xc$ of type $\typecox$}\big\}
    & \longrightarrow \big\{\text{Cox rings of $\Xc$ of type $\typecox$}\big\}\\
    \coxsheaf\quad& \longmapsto\quad \coxsheaf(\Xc)
  \end{align*}
  is an equivalence of categories.
\end{cor}

\begin{proof}
  Let $\coxsheaf$ be a Cox sheaf of $\Xc$ of type $\typecox$. By
  Definition \ref{def:cox_sheaf_nonclosed}, $\coxsheafbar$ is a
  $\galois$-equivariant Cox sheaf of $\Xcbar$. Hence
  $\coxsheafbar(\Xcbar)$ is a Cox ring of $\Xcbar$ of type $\typecox$,
  and the induced $\galois$-action turns it into a
  $\galois$-equivariant Cox ring of type $\typecox$. Then the subring
  of $\galois$-invariant elements
  $\coxsheaf(\Xc)=\coxsheafbar(\Xcbar)^\galois$ is a Cox ring of $\Xc$
  of type $\typecox$ by Proposition
  \ref{prop:galois_descent_cox_rings_sheaves}.
    
  Assume now that $\gradgroup=\gradgroupeff$. The functor is
  essentially surjective by Propositions
  \ref{prop:cox_ring_cox_sheaf}, \ref{prop:equiv_cox_ring_sheaf} and
  \ref{prop:galois_descent_cox_rings_sheaves}. We prove that it is
  fully faithful.  Let $\coxsheaf$ and $\coxsheaf'$ be Cox sheaves of
  $\Xc$ of type $\typecox$. A morphism $\morphism\colon \coxsheaf(\Xc) \to
  \coxsheaf'(\Xc)$ of Cox rings of $\Xc$ of type $\typecox$ induces a
  morphism of $\galois$-equivariant Cox rings
  $\coxsheaf(\Xc)_{\fieldbar} \to \coxsheaf'(\Xc)_{\fieldbar}$ that
  extends to a unique isomorphism of Cox sheaves
  $\overline\morphism\colon \coxsheafbar\to\coxsheafbar'$ by Propositions
  \ref{prop:morphism_cox_ring_cox_sheaf} and
  \ref{prop:morphism_cox_sheaves}. Since $\galoiselement^{-1} \circ
  \overline\morphism^{-1} \circ \galoiselement \circ
  \overline\morphism$ is an automorphism on $\coxsheafbar$ that
  induces the identity on $\coxsheafbar(\Xcbar)$ for every
  $\galoiselement \in \galois$, by the uniqueness statement of
  Proposition~\ref{prop:morphism_cox_ring_cox_sheaf},
  $\galoiselement^{-1} \circ \overline\morphism^{-1} \circ
  \galoiselement \circ \overline\morphism$ is the identity as well,
  which means that $\overline\morphism$ is $\galois$-equivariant, and
  hence restricts to an isomorphism $\coxsheaf\cong\coxsheaf'$
  inducing $\morphism$.
\end{proof}

We apply descent theory (see \cite[\S III.1]{MR1867431} or \cite[\S
16]{Milne}) to classify Cox rings and Cox sheaves of $\Xc$ up to isomorphism.
Recall that by Corollary~\ref{cor:existence_isom_cox_sheaves}, every two Cox
sheaves of $\Xc$ of type $\typecox$ become isomorphic after the field
extension $\field\subseteq\fieldbar$. Moreover, if $\morphism$ is an
isomorphism of Cox sheaves of type $\typecox$ between two
$\galois$-equivariant Cox sheaves of $\Xcbar$ of type $\typecox$, then
${^\galoiselement}\morphism\coloneqq \galoiselement\circ\morphism\circ\galoiselement^{-1}$
is another isomorphism for all $\galoiselement\in\galois$, and sending
$\galoiselement$ to the automorphism
$\morphism^{-1} \circ {^\galoiselement}\morphism$ defines a map
$\galois \to \gradgroupdual(\fieldbar)$ by Proposition
\ref{prop:cox_sheaf_aut}, which turns out to be a $1$-cocycle.

\begin{prop}\label{prop:cox_sheaf_classification}
  Assume that $\Xc$ has a Cox sheaf $\coxsheaf$ of type $\typecox$.  Sending a
  Cox sheaf $\coxsheaf'$ of $\Xc$ of type $\typecox$ with a Cox sheaf
  isomorphism $\morphism\colon  \coxsheafbar \to \coxsheafbar'$ to the class of the
  cocycle
  \begin{equation*}
    \galois \to \Aut(\coxsheafbar) \cong \gradgroupdual(\fieldbar),
    \quad \galoiselement\mapsto \morphism^{-1} \circ {^\galoiselement}\morphism
  \end{equation*}
  defines an bijective map  from the set of isomorphism classes of Cox
  sheaves of $\Xc$ of type $\typecox$ to $H^1_{\text{\it{\'et}}}(\field, \gradgroupdual)$. 
\end{prop}

\begin{proof}
  The class of the cocycle
  $\galoiselement \mapsto \morphism^{-1} \circ {^\galoiselement}\morphism$
  does not depend on the choice of $\morphism$. Moreover, if $\coxsheaf',
  \coxsheaf''$ are two Cox sheaves of $\Xc$ of type $\typecox$, and
  $\morphism'\colon \coxsheafbar \to \coxsheafbar'$,
  $\morphism''\colon \coxsheafbar \to \coxsheafbar''$ are two Cox sheaf
  isomorphisms, the associated cocycles have the same class in
  $H^1_{\text{\it{\'et}}}(\field, \gradgroupdual)$ if and only if there is an
  isomorphism of Cox sheaves $\coxsheaf'\to\coxsheaf''$. Hence the map in the
  statement is injective.

  For surjectivity, note that a cocycle
  $\galoiscocycle \colon  \galois \to \gradgroupdual(\fieldbar)$ defines a \emph{twisted
  action}
  \begin{equation*}
    \galois \times \coxsheafbar \to \coxsheafbar, \qquad
    (\galoiselement, \coxsec) \mapsto
    \galoiscocycle_\galoiselement(\galoiselement(\coxsec))\coloneqq 
    (\morphism_{\galoiscocycle_\galoiselement} \circ
    (\identity_{\coxsheaf} \otimes \galoiselement))(\coxsec),
  \end{equation*}
  where $\morphism_{\galoiscocycle_\galoiselement}$ is the automorphism
  defined in Proposition \ref{prop:cox_sheaf_aut}.  Since
  \begin{equation*}
    H^1(\galois,\gradgroupdual(\fieldbar))
    =\varinjlim_{\fieldextension/\field}
    H^1(\gal(\fieldextension/\field),\gradgroupdual(\fieldextension)),
  \end{equation*}
  where the direct limit is taken over the finite Galois extensions
  $\fieldextension$ of $\field$ inside $\fieldbar$, the twisted action defined
  by $\galoiscocycle$ is continuous, and hence it is a {\compatible}
  $\galois$-action on $\coxsheafbar$ according to Definition
  \ref{def:galois_action}.  Let $\coxsheaf^\galoiscocycle$ be the sheaf of
  invariants of this action.  By Proposition
  \ref{prop:galois_descent_cox_rings_sheaves}, $\coxsheaf^\galoiscocycle$ is a
  Cox sheaf of $\Xc$ of type $\typecox$ such that
  $\coxsheafbar\cong\coxsheaf^\galoiscocycle\otimes_\field\fieldbar$. The
  cocycle associated with this isomorphism is
  $(\identity_{\coxsheafbar}\circ\morphism_{\galoiscocycle_\galoiselement}\circ
  \galoiselement\circ\identity_{\coxsheafbar}\circ
  \galoiselement^{-1})_{\galoiselement\in \galois}=\galoiscocycle$.
\end{proof}

As a consequence of Corollary \ref{cor:cox_ring_cox_sheaf_nonclosed} and
Proposition \ref{prop:cox_sheaf_classification}, the map in Proposition
\ref{prop:cox_sheaf_classification} defines a bijection between the set of
isomorphism classes of Cox rings of $\Xc$ of type $\typecox$ and
$H^1_{\text{\it{\'et}}}(\field, \gradgroupeffdual)$.

The inverse map to the bijection in Proposition
\ref{prop:cox_sheaf_classification} is obtained by twisting a Cox sheaf of
$\Xc$ of type $\typecox$ by cocycles. Therefore, we introduce the notion of
twisted Cox sheaf.

\begin{defin}\label{def:twisted_cox_sheaf}
  For every Cox sheaf $\coxsheaf$ of $\Xc$ of type $\typecox$ and every
  cocycle $\galoiscocycle\colon  \galois \to \gradgroupdual(\fieldbar)$, we denote
  by $\twistedcoxsheaf{\galoiscocycle}$ the \emph{twisted Cox sheaf}
  constructed in the proof of Proposition \ref{prop:cox_sheaf_classification}.
\end{defin}

\begin{prop}\label{prop:cox_sheaves_torsors_nonclosed}
  If $\coxsheaf$ is a Cox sheaf of $\Xc$ of type $\typecox$, then
  $\spec_{\Xc}\coxsheaf$ is an $\Xc$-torsor of type
  $\typecox$. Moreover, for each cocycle $\galoiscocycle \colon  \galois \to
  \gradgroupdual(\fieldbar)$, the $\Xc$-torsor
  $\spec_{\Xc}\twistedcoxsheaf{\galoiscocycle}$ has class
  $[\spec_{\Xc}\coxsheaf]-[\galoiscocycle]$ in
  $H^1_{\text{\it{\'et}}}(\Xc, \gradgroupdualX)$.
\end{prop}

\begin{proof}
  Let $\coxsheaf$ be a Cox sheaf of $\Xc$ of type $\typecox$. Since
  $\coxsheafbar$ is a $\galois$-equivariant Cox sheaf of $\Xcbar$, the
  action of $\spec\fieldbar[\gradgroup]$ on
  $\spec_{\Xcbar}\coxsheafbar$ induced by the $\gradgroup$-grading on
  $\coxsheafbar$ descends to an action of
  $\gradgroupdual=\spec\fieldbar[\gradgroup]^{\galois}$ on
  $\spec_{\Xc}\coxsheaf$, and the canonical morphism
  $\spec_{\Xc}\coxsheaf\to\Xc$ is an $\Xc$-torsor of type $\typecox$
  by Proposition \ref{prop:cox_sheaves_torsors} and
  {\it{fpqc}}-descent.

  Let $\galoiscocycle \colon  \galois \to\gradgroupdual(\fieldbar)$ be a cocycle. By
  Proposition \ref{prop:cox_sheaf_classification}, the twisted Cox sheaf
  $\twistedcoxsheaf{\galoiscocycle}$ corresponds to the $\galois$-equivariant
  Cox sheaf $\coxsheafbar$ of $\Xcbar$ with the twisted action of $\galois$
  given by
  $(\galoiselement, \coxsec) \mapsto
  \galoiscocycle_\galoiselement(\galoiselement(\coxsec))$ under the bijection
  of Proposition \ref{prop:galois_descent_cox_rings_sheaves}.  Thus,
  $\spec_{\Xc}\twistedcoxsheaf{\galoiscocycle}$ is obtained by Galois descent
  from $\spec_{\Xcbar}\coxsheafbar$ with the twisted $\galois$-action
  $(\galoiselement,x)\mapsto
  \galoiscocycle_\galoiselement^{-1}(\galoiselement(x))$. Therefore,
  $[\spec_{\Xc}\twistedcoxsheaf{\galoiscocycle}]
  =[\spec_{\Xc}\coxsheaf]-[\galoiscocycle]$ in
  $H^1_{\text{\it{\'et}}}(\Xc, \gradgroupdualX)$; see \cite[Example 2,
  p.~21]{MR1845760}.
\end{proof}

\begin{prop}\label{prop:torsors_cox_sheaves_nonclosed}
  If $\torsormor\colon  \Yc \to \Xc$ is an $\Xc$-torsor of type $\typecox$, then
  $\torsormor_*\OO_\Yc$ is a Cox sheaf of $\Xc$ of type $\typecox$,
  $\OO_\Yc(\Yc)$ is a Cox ring of $\Xc$ of type $\typecox$, and
  $\spec_\Xc\pi_*\OO_\Yc\cong \Yc$.
\end{prop}

\begin{proof}
  Let $\torsormor\colon  \Yc \to \Xc$ be an $\Xc$-torsor of type
  $\typecox$. Let $\torsormorbar\colon  \Ycbar \to \Xcbar$ be the induced
  $\Xcbar$-torsor of type $\typecox$. Then
  $\torsormorbar_*\OO_{\Ycbar}$ is a Cox sheaf of $\Xcbar$ of type
  $\typecox$ by Proposition \ref{prop:cox_sheaves_torsors}.  The
  induced $\galois$-action on $\torsormorbar_*\OO_{\Ycbar}$ is a
  natural $\galois$-action in the sense of Definition
  \ref{def:galois_action}, as it is continuous and compatible with the
  $\galois$-action on $\gradgroup$. Hence $\torsormor_*\OO_\Yc$ is a
  Cox sheaf of $\Xc$ of type $\typecox$ according to
  Definition~\ref{def:cox_sheaf_nonclosed}, and $\OO_\Yc(\Yc) =
  (\torsormor_*\OO_\Yc)(\Xc)$ is a Cox ring of $\Xc$ of type
  $\typecox$ by Corollary \ref{cor:cox_ring_cox_sheaf_nonclosed}.
  Moreover, $\torsormor$ is affine by \cite[Proposition
  0.7]{MR1304906}. Hence $\spec_\Xc\torsormor_*\OO_\Yc\cong \Yc$.
\end{proof}

\begin{cor}\label{cor:bijection_torsors_cox_sheaves}
  The contravariant functor
  \begin{align*}
    \big\{\text{Cox sheaves of $\Xc$ of type $\typecox$}\big\}
    & \longrightarrow \big\{\text{$\Xc$-torsors of type $\typecox$}\big\}\\
    \coxsheaf\quad
    & \longmapsto\quad \spec_{\Xc}\coxsheaf
  \end{align*}
  is an anti-equivalence of categories, with inverse functor
  \begin{align*}
    \big\{\text{$\Xc$-torsors of type $\typecox$}\big\}
    & \longrightarrow \big\{\text{Cox sheaves of $\Xc$ of type $\typecox$}\big\}\\
    \torsormor\colon \Yc\to \Xc\quad
    & \longmapsto\quad \torsormor_*\OO_\Yc.
  \end{align*}
\end{cor}
\begin{proof}
  By Propositions~\ref{prop:cox_sheaves_torsors_nonclosed} and
  \ref{prop:torsors_cox_sheaves_nonclosed}, both functors are well-defined and
  essentially surjective. Moreover, the functor $\spec_{\Xc}$ is fully
  faithful.
\end{proof}

\begin{proof}[Proof of Theorem \ref{introduction:theorem:classification}]
  Combine Corollary \ref{cor:cox_ring_cox_sheaf_nonclosed} with Corollary
  \ref{cor:bijection_torsors_cox_sheaves}.
\end{proof}

Now we discuss the functorial properties of Cox rings and Cox sheaves
of varieties over an arbitrary field. We refer to \cite[\S 1.5]{MR899402} and
\cite[\S 6]{MR1995130} for the analogous properties of torsors under
quasitori and of Cox rings of identity type, respectively.

By \cite[Proposition 1.5.2]{MR899402}, the exact
sequence~(\ref{eq:exact_sequence_type}), which classifies $X$-torsors under a
quasitorus $G$, is functorial in $\Xc$ and $\group$, covariant in $\group$ and
contravariant in $\Xc$. Therefore, if $M$ denotes the $\galois$-module dual to
$\group$ (i.e., the group of characters of $G_{\kbar}$), the exact
sequence~(\ref{eq:exact_sequence_type}) is functorial in $\Xc$ and
$\gradgroup$, contravariant with respect to both.

We consider first the functoriality with respect to morphisms of quasitori
$\group\to\group'$, or equivalently, with respect to the induced morphisms of
dual $\galois$-modules $\groupmorphism\colon \gradgroup'\to\gradgroup$.

If
$\coxsheaf$ is a Cox sheaf of $\Xc$ of type
$\typecox\colon \gradgroup\to\pic(\Xcbar)$, and
$\{\coxisom_{\groupelement,
  \divisor}\}_{(\groupelement,\divisor)\in\gradgroupdiv}$ is a family of
isomorphisms associated with $\coxsheaf_{\kbar}$, then the family of
isomorphisms
$\{\coxisom_{\groupmorphism(\groupelement'),
  \divisor}\}_{(\groupelement',\divisor)\in\gradgroup'_{\lambda\circ\varphi}}$
defines on
$\bigoplus_{\groupelement'\in\gradgroup'}(\coxsheaf_{\kbar})_{\groupmorphism(\groupelement')}$ a structure of Cox sheaf of type $\lambda\circ\varphi$. The
latter inherits from $\coxsheaf_{\kbar}$ an action of $\galois$ that turns it
into a $\galois$-equivariant Cox sheaf.
\begin{defin}\label{def:pullback}  
  We define the \emph{pullback of $\coxsheaf$ under $\varphi$} to be
  \begin{equation*}
  \groupmorphism^*\coxsheaf\coloneqq 
  \left(\bigoplus_{\groupelement'\in\gradgroup'}
  (\coxsheaf_{\kbar})_{\groupmorphism(\groupelement')}\right)^{\galois}.
\end{equation*}

\end{defin}

The morphism
$\groupmorphism^*\coxsheafbar\to\coxsheafbar$ that restricts to the identity
$(\groupmorphism^*\coxsheafbar)_{\groupelement'}
=(\coxsheafbar)_{\groupmorphism(\groupelement')}$ for all
$\groupelement'\in\gradgroup'$ is $\galois$-equivariant and descends to a
morphism of $\OO_{\Xc}$-algebras
\begin{equation}\label{eq:natural_morphism_pullback}
  \groupmorphism^*\coxsheaf\to\coxsheaf.
\end{equation} 
Moreover, each morphism $\morphism\colon \coxsheaf\to\coxsheaf'$ of Cox sheaves of
$\Xc$ of type $\typecox$ pulls back under $\groupmorphism$ to a morphism
$\groupmorphism^*(\morphism)
\colon \groupmorphism^*\coxsheaf\to\groupmorphism^*\coxsheaf'$ of Cox sheaves of
$\Xc$ of type $\typecox\circ\groupmorphism$ such that the following diagram is
commutative:
\begin{equation*}
  \begin{tikzpicture}
    \matrix (m) [matrix of math nodes, row sep=2em,
    column sep=3.5em, text height=1.5ex, text depth=0.25ex]
    { \groupmorphism^*\coxsheaf & \groupmorphism^*\coxsheaf' \\
      \coxsheaf & \coxsheaf'. \\ };
    
    \path[-stealth] (m-1-1) edge node[auto]
    {$\groupmorphism^*(\morphism)$}
    (m-1-2) edge node [auto] {} (m-2-1) 
    (m-2-1) edge node [auto]
    {$\morphism$} (m-2-2)  
    (m-1-2) edge node [auto] {} (m-2-2);
  \end{tikzpicture}
\end{equation*}

We show that the pullback of Cox sheaves under $\groupmorphism$
corresponds to the pushforward of torsors defined in \cite[Example 3,
p.~21]{MR1845760} under the dual morphism of quasitori
$\groupmorphismdual\colon \group\to\group'$.

\begin{prop}
  Let $\varphi\colon \gradgroup'\to\gradgroup$ be a morphism of finitely
  generated $\galois$-modules whose torsion is coprime to the
  characteristic of $\field$. Let $\coxsheaf$ be a Cox sheaf of type
  $\typecox\colon \gradgroup\to\pic(\Xcbar)$. Then
  $\groupmorphismdual_*\spec_\Xc\coxsheaf\cong\spec_\Xc\groupmorphism^*\coxsheaf$.
\end{prop}

\begin{proof}
  Let $\{U_i\}_i$ be an affine \'etale covering of $\Xc$ that
  trivializes the torsor $\spec_X\coxsheaf$. Let $U_{i,j}\coloneqq U_i\times_X U_j$, and let
  \begin{equation*}
    \cocycleY_{i,j}\in
    \Hom_{\galois}(\gradgroup,\OO_{U_{i,j,\fieldbar}}(U_{i,j,\fieldbar})^\times)
    \cong\gradgroupdual(U_{i,j})
  \end{equation*}
  be a cocycle representing $[\spec_\Xc\coxsheaf]\in
  H^1_{\text{\it{\'et}}}(\Xc,\gradgroupdualX)$.  Then the class
  $\groupmorphismdual_*[(\cocycleY_{i,j})_{i,j}]$ of
  $\groupmorphismdual_*\spec_\Xc\coxsheaf$ in
  $H^1_{\text{\it{\'et}}}(\Xc,\widehat{\gradgroup'}_{\Xc})$ is
  represented by the cocycle
  $(\cocycleY_{i,j}\circ\groupmorphism)_{i,j}$.

  By definition, the torsor $\spec_\Xc\coxsheaf$ is obtained by gluing
  \begin{equation*}
    \{(\spec_\Xc\coxsheaf)\times_\Xc
    U_{i}\cong \spec_{U_{i}}\OO_{U_{i,\fieldbar}}[M]^\galois\}_{i}
  \end{equation*}
  over the isomorphisms
  \begin{equation*}
    \OO_{U_{j,i,\fieldbar}}[M]
    \to\OO_{U_{i,j,\fieldbar}}[M],\qquad
    1\cdot m\mapsto\cocycleY_{i,j}(m)\cdot \groupelement, \quad \forall m\in M.
  \end{equation*}
  By definition of pullback, $\spec_\Xc\groupmorphism^*\coxsheaf$ is
  obtained by gluing
  \begin{equation*}
    \{(\spec_\Xc\groupmorphism^*\coxsheaf)\times_\Xc
    U_{i}\cong \spec_{U_{i}}\OO_{U_{i,\fieldbar}}[M']^\galois\}_{i}
  \end{equation*}
  over the isomorphisms
  \begin{equation*}
    \OO_{U_{j,i,\fieldbar}}[M']
    \to\OO_{U_{i,j,\fieldbar}}[M'],\qquad
    1\cdot m'\mapsto\cocycleY_{i,j}(\groupmorphism(m'))\cdot \groupelement',
    \quad \forall m'\in M'.
  \end{equation*}
  Hence the isomorphism class of $\spec_\Xc\groupmorphism^*\coxsheaf$
  in $H^1_{\text{\it{\'et}}}(\Xc,\widehat{\gradgroup'}_{\Xc})$ is
  represented by the cocycle
  $(\cocycleY_{i,j}\circ\groupmorphism)_{i,j}$.
\end{proof}

Now we consider the functoriality with respect to morphisms of
$\field$-varieties. Let $\structuremor\colon \Xc'\to\Xc$ be a morphism of
$\field$-varieties and $\coxsheaf$ a Cox sheaf of $\Xc$ of type
$\typecox\colon \gradgroup\to\pic(\Xcbar)$. We denote by
$\structuremor^*\colon \pic(\Xcbar)\to\pic(\Xcbar')$ the pullback of divisor classes
under $\structuremor$. We recall that the property of being a torsor over
$\Xc$ under $\group$ is stable under base extension. Therefore,
$\spec_{\Xc'}\structuremor^*\coxsheaf\cong\Xc'\times_\Xc\spec_\Xc\coxsheaf$ is
an $\Xc'$-torsor under $\gradgroupdual$.

\begin{prop}\label{prop:pullback_under_morphism_varieties}
  Let $\structuremor\colon \Xc'\to\Xc$ be a morphism of $\field$-varieties,
  $\structuremor^*\colon \pic(\Xcbar)\to\pic(\Xcbar')$ the induced pullback of
  divisor classes, and $\coxsheaf$ a Cox sheaf of $\Xc$ of type
  $\typecox\colon \gradgroup\to\pic(\Xcbar)$.
  Then $\structuremor^*\coxsheaf$ is a Cox sheaf of $\Xc'$ of type
  $\structuremor^*\circ\typecox\colon \gradgroup\to\pic(\Xcbar')$.
\end{prop}

\begin{proof}
  If
  $\{\coxisom_{\groupelement,
    \divisor}\}_{(\groupelement,\divisor)\in\gradgroupdiv}$ is a family of
  isomorphisms associated with the Cox sheaf $\coxsheaf$, then the family of
  isomorphisms
  \begin{equation*}
    \structuremor^*(\coxisom_{\groupelement, \divisor})\colon 
    \structuremor^*\coxsheaf_{\groupelement}\to\structuremor^*\OO_{\Xc'}(\divisor),
  \end{equation*}
  for $(\groupelement,\divisor)\in\gradgroupdiv$, defines a structure of Cox
  sheaf of $\Xc'$ of type $\structuremor^*\circ\typecox$ on
  $\structuremor^*\coxsheaf$, as $\structuremor^*$ commutes with direct
  limits.
\end{proof}

\begin{cor}\label{cor:functoriality}
  Let $\structuremor\colon \Xc'\to\Xc$ be a morphism of $\field$-varieties,
  and $\groupmorphism\colon \gradgroup'\to\gradgroup$ a morphism of
  finitely generated $\galois$-modules with torsion coprime to the
  characteristic of $\field$.
  \begin{enumerate}[label=(\arabic{*}), ref={\it{(\arabic{*})}}]
  \item\label{item:double_pullback_cox_sheaves} If $\coxsheaf$ is a
    Cox sheaf of $\Xc$ of type $\typecox\colon \gradgroup\to\pic(\Xcbar)$,
    then $\structuremor^*\groupmorphism^*\coxsheaf$ is a Cox sheaf of
    $\Xc'$ of type
    \begin{equation*}
      \structuremor^*\circ\typecox\circ\groupmorphism\colon \gradgroup'\to\pic(\Xcbar'). 
    \end{equation*}
  \item\label{item:double_pullback_morphism} If
    $\morphism\colon \coxsheaf\to\coxsheaf'$ is a morphism of Cox sheaves of
    $\Xc$ of type $\typecox$, then
    \begin{equation*}
      \structuremor^*(\groupmorphism^*(\morphism))\colon 
      \structuremor^*\groupmorphism^*\coxsheaf\to\structuremor^*\groupmorphism^*\coxsheaf'
    \end{equation*}
    is a morphism of Cox sheaves of $\Xc'$ of type
    $\structuremor^*\circ\typecox\circ\groupmorphism$.
  \item\label{item:double_pullback_functor} The double pullback
    $\structuremor^*\circ\groupmorphism^*$ is a covariant functor from
    the category of Cox sheaves of $\Xc$ of type $\typecox$ to the
    category of Cox sheaves of $\Xc'$ of type
    $\structuremor^*\circ\typecox\circ\groupmorphism$.
  \item There is an isomorphism of functors between
    $\structuremor^*\circ\groupmorphism^*$ and
    $\groupmorphism^*\circ\structuremor^*$.\label{item:pullback_functor_isomorphism}
  \end{enumerate}
\end{cor}

\begin{proof}
  Parts \ref{item:double_pullback_cox_sheaves} and
  \ref{item:double_pullback_morphism} are a consequence of Proposition
  \ref{prop:pullback_under_morphism_varieties} and the discussion
  above.

  For \ref{item:double_pullback_functor}, we observe that
  $\groupmorphism^*$ and $\structuremor^*$ are both covariant
  functors.  For \ref{item:pullback_functor_isomorphism}, we observe
  that
  $\structuremor^*\groupmorphism^*\coxsheaf=\groupmorphism^*\structuremor^*\coxsheaf$
  as $\structuremor^*$ commutes with direct limits. Therefore, also
  $\structuremor^*(\groupmorphism^*(\morphism))=\groupmorphism^*(\structuremor^*(\morphism))$,
  as $\groupmorphism^*(\morphism)$ and
  $\groupmorphism^*(\structuremor^*(\morphism))$ restrict to
  $\morphism$ and $\structuremor^*(\morphism)$, respectively, on the
  homogeneous parts of the Cox sheaves.
\end{proof}

The analogous functoriality properties of Cox rings can be deduced from
Corollary \ref{cor:functoriality} via Corollary
\ref{cor:cox_ring_cox_sheaf_nonclosed}.

\begin{remark}
  As in Corollary \ref{cor:functoriality}, let
  $\structuremor\colon \Xc'\to\Xc$ be a morphism of $\field$-varieties,
  assume that $\Xc'$ has a Cox sheaf $\coxsheaf'$ of identity type,
  and let $\coxsheaf$ be a Cox sheaf of $\Xc$ of type $\typecox$. Up
  to twisting $\coxsheaf'$ by an element of
  $H^1_{\text{\it{\'et}}}(\field,\gradgroupdual)$ as in Proposition~\ref{prop:cox_sheaf_classification}, we can assume that
  $\structuremor^*\coxsheaf\cong(\structuremor^*\circ\typecox)^*\coxsheaf'$. Then
  there exists a natural morphism of graded $\OO_{\Xc'}$-algebras
  $\structuremor^*\coxsheaf\to\coxsheaf'$ as in
  \eqref{eq:natural_morphism_pullback}, and also a morphism of graded
  $\OO_{\Xc}$-algebras $\coxsheaf\to\structuremor_*\coxsheaf'$, as
  $\structuremor^*$ and $\structuremor_*$ are adjoint functors. This
  gives a new proof of the existence statement in \cite[Proposition
  5.3]{MR1995130} by considering $\field=\fieldbar$ and
  $\typecox=\identity_{\pic(\Xc)}$.  We observe that in general
  $\structuremor_*\coxsheaf'$ is not a Cox sheaf of $\Xc$.
\end{remark}

Now we turn to existence criteria for Cox rings and Cox sheaves of varieties
over nonclosed fields.  We start with an example that shows that Cox rings of
identity type do not always exist.  We also observe that a Cox ring of $\Xc$
of type $\lambda$ exists if $\lambda$ factors through
$\pic(\Xc) \hookrightarrow \pic(\Xcbar)$, as Construction
\ref{construction:cox_sheaf_abstract} can be performed over $k$.

\begin{example}\label{exa:nonexistence}
  Let $\Xc \subset \PP^2_\RR$ be the conic defined by $x^2+y^2+z^2=0$,
  with $\Xc(\RR)=\emptyset$. The base change $\Xc_{\CC}$ is the image
  of the closed immersion
  \begin{equation*}
    \projectivemorphism\colon  \PP^1_\CC \to \PP^2_\CC, \qquad (u:v) \mapsto
    (2uv:u^2-v^2:\imaginary(u^2+v^2)).
  \end{equation*}
  A Cox ring of $\PP^1_\CC$ of type $\identity_{\pic(\PP^1_\CC)}$ is
  $\CC[u,v]$, where $u,v \in H^0(\PP^1_\CC, \mathcal{O}(1))$ vanish in
  $(0:1),(1:0)$, respectively. If $\Xc$ has a Cox ring $\coxring$ of
  type $\identity_{\pic(\Xc_{\CC})}$ over $\RR$, then $\coxring_\CC$
  is endowed with a natural action of
  $\galois\coloneqq \gal(\CC/\RR)=\{\identity,\galoiselement\}$. Via the
  isomorphism $\projectivemorphism$, this action pulls back to an
  action of $\galois$ on the Cox ring $\CC[u,v]$ of $\PP^1_\CC$ with
  the following properties.
  There exists
  $\galoisconstant\in \CC^\times$ such that
  $\galoiselement(u)=\galoisconstant v$ since $\galoiselement$ exchanges
  $(0:1:\imaginary)=\projectivemorphism((1:0))$ and
  $(0:1:-\imaginary)=\projectivemorphism((0:1))$.
  Then $\galoiselement^2(u)=u$
  gives $\galoiselement(v)=\galoiselement(\galoisconstant)^{-1}u$.
  Furthermore, $\galoiselement$ exchanges
  the points $(1:0:\imaginary)=\projectivemorphism((1:1))$ and
  $(1:0:-\imaginary)=\projectivemorphism((1:-1))$, hence
  $\galoiselement(u-v)=\galoisconstant
  v-\galoiselement(\galoisconstant)^{-1}u$ should be a scalar multiple
  of $u+v$. This implies
  $\galoisconstant=-\galoiselement(\galoisconstant)^{-1}$, which is
  impossible for $\galoisconstant\in \CC^\times$.  Therefore, the
  conic $X$ without $\RR$-rational points does not have a Cox ring of
  type $\identity_{\pic(\Xc_{\CC})}$.
  
  Now we compute a Cox ring of $\Xc$ of type
  $\pic(\Xc)\hookrightarrow\pic(\Xc_{\CC})$. Since $\pic(\Xc)$ is the subgroup of
  index 2 of $\pic(\Xc_{\CC})$, by Definition \ref{def:pullback} a Cox ring of
  $\Xc_{\CC}$ of type $\pic(\Xc)\hookrightarrow\pic(\Xc_{\CC})$ is isomorphic via
  $\psi$ to the subring of $\CC[u,v]$ generated by $u^2, uv, v^2$. Since
  $\CC[u^2, uv, v^2]=\CC[2uv,u^2-v^2,\imaginary(u^2+v^2)]$ as subrings of
  $\CC[u,v]$, we see that the homogeneous coordinate ring
  $\CC[x,y,z]/(x^2+y^2+z^2)$ of $X_{\CC}$ is a $\galois$-equivariant Cox ring
  of $\Xc_{\CC}$ of type $\pic(\Xc)\hookrightarrow\pic(\Xc_{\CC})$.  Hence
  $\RR[x,y,z]/(x^2+y^2+z^2)$ is a Cox ring of $\Xc$ of type
  $\pic(\Xc)\hookrightarrow\pic(\Xc_{\CC})$ in the sense of Definition
  \ref{def:cox_sheaf_nonclosed}.
\end{example}

\begin{remark}\label{rem:pullback_short}
  Assume that $\Xc$ has a Cox sheaf (ring) of injective type
  $\gradgroup'\hookrightarrow\pic(\Xcbar)$, where $\gradgroup'$ is a
  $\galois$-invariant subgroup of $\pic(\Xcbar)$. If
  $\typecox \colon  \gradgroup \to \pic(\Xcbar)$ factors through $\gradgroup'$, Cox
  sheaves (rings) of $\Xc$ of type $\typecox$ exist by pullback (see
  Definition \ref{def:pullback}).  Therefore, if $\Xc$ has a Cox sheaf (ring)
  of identity type, then Cox sheaves (rings) of $\Xc$ of all types exist
  (cf.~\cite[p.~25]{MR1845760}).
\end{remark}

The next proposition relates the existence of Cox sheaves to the existence of
$\galois$-equivariant characters in Construction
\ref{construction:cox_sheaf_abstract}, and shows that our definition of Cox
sheaves of identity type over nonclosed fields recovers \cite[Construction
6.1.3.3]{arXiv:1003.4229} for locally factorial $X$.
 
\begin{prop}\label{prop:equivariant_character}
  Let $\gradgroupgenerator\subseteq\gradgroup$ be a finite $\galois$-invariant
  set of generators for $\gradgroup$, and $\divisorsgenerator$ a finite
  $\galois$-invariant set of Cartier divisors such that
  $\typecox(\gradgroupgenerator)=\{[\divisor]:\divisor\in\divisorsgenerator\}$. Let
  \begin{equation*}
    \gradgroupgeneratordiv\coloneqq \{(\groupelement,\divisor)
    \in\gradgroupgenerator\times\divisorsgenerator:[\divisor]
    =\typecox(\groupelement)\}
  \end{equation*}
  with the componentwise action of $\galois$, and let
  $\freegroup\coloneqq \bigoplus_{\freeelement\in\gradgroupgeneratordiv}\ZZ\freeelement$
  with the induced $\galois$-action. Let
  $\presentation\colon \freegroup\to\gradgroup$ be defined by
  $\presentation((\groupelement,\divisor))\coloneqq \groupelement$ for all
  $(\groupelement,\divisor)\in\freegroup$, and
  $\freealgebra\coloneqq \bigoplus_{(\groupelement,\divisor)\in\freegroup}\OO_{\Xcbar}(\divisor)$
  the associated $\OO_{\Xcbar}$-algebra as in Construction
  \ref{construction:cox_sheaf_abstract}.  Then a Cox sheaf of $\Xc$ of type
  $\typecox$ exists if and only if there is a $\galois$-equivariant character
  associated with $\freealgebra$.
 \end{prop}

 \begin{proof}
   Let $\kernfreegroup$ be the kernel of $\presentation$, and let
   $\freealgebra_{(\groupelement,\divisor)}\coloneqq \OO_{\Xcbar}(\divisor)$
   for all $(\groupelement,\divisor)\in\freegroup$.  Given a
   $\galois$-equivariant character
   $\character\colon \kernfreegroup\to\FfieldXbar^\times$ associated with
   $\freealgebra$, the sheaf $\freeideal$ of ideals of $\freealgebra$
   defined by $\character$ is invariant under the action of $\galois$
   on $\freealgebra$ induced by the natural $\galois$-action on
   $\FfieldXbar$. Therefore, $\coxsheaf\coloneqq \freealgebra/\freeideal$ is a
   $\galois$-equivariant Cox sheaf of $\Xcbar$ of type $\typecox$, and
   the sheaf of invariants $\coxsheaf^\galois$ is a Cox sheaf of $\Xc$
   of type $\typecox$.

   Assume now that there is a Cox sheaf $\coxsheaf$ of $\Xc$ of type
   $\typecox$. Then $\coxsheafbar$ is a $\galois$-equivariant Cox
   sheaf of $\Xcbar$ of type $\typecox$. By Proposition
   \ref{prop:isomorphism_cox_sheaves_abstract}, we can assume that
   $\coxsheafbar=\freealgebra/\freeideal$, where $\freeideal$ is the
   sheaf of ideals of $\freealgebra$ defined by a character
   $\character\colon \kernfreegroup\to\FfieldXbar^\times$ associated with
   $\freealgebra$. Let $\freemor\colon \freealgebra\to\coxsheafbar$ be the
   projection.
 
   Let
   $\{\coxisom_{\groupelement,\divisor}\}_{(\groupelement,\divisor)\in\gradgroupdiv}$
   be a family of isomorphisms associated with $\coxsheafbar$. Without
   loss of generality, we can assume that
   $\coxisom_{\groupelement,\divisor}=\freemor|_{\freeinvsheaf{(\groupelement,\divisor)}}^{-1}$
   for all $(\groupelement,\divisor)\in\freegroup$.  By Definition
   \ref{def:galois_action} and Lemma \ref{lem:isom_inv_sheaves}, for
   every $\galoiselement\in \galois$ and every
   $(\groupelement,\divisor)\in\freegroup$, there exists a constant
   $\compconstant_{\galoiselement,(\groupelement,\divisor)}\in\fieldbar^\times$
   such that
  \begin{equation*}
    (\freemor|_{\freeinvsheaf{(\galoiselement\groupelement,\galoiselement\divisor)}}^{-1}
    \circ
    \galoiselement\circ\freemor|_{\freeinvsheaf{(\groupelement,\divisor)}})(\coxsec)
    =\compconstant_{\galoiselement,(\groupelement,\divisor)}\galoiselement(\coxsec)
  \end{equation*}
  for all sections $\coxsec$ of
  $\freeinvsheaf{(\groupelement,\divisor)}=\OO_{\Xcbar}(\divisor)$.
  These constants satisfy
  \begin{equation*}
    \compconstant_{\galoiselement,\freeelement+\freeelement'}
    =\compconstant_{\galoiselement,\freeelement}\compconstant_{\galoiselement,\freeelement'} \qquad
    \compconstant_{\galoiselement\galoiselement',\freeelement}
    =\compconstant_{\galoiselement,\galoiselement'(\freeelement)}
    \galoiselement(\compconstant_{\galoiselement',\freeelement})
  \end{equation*}
  for all $\freeelement, \freeelement'\in\freegroup$ and all
  $\galoiselement,\galoiselement'\in\galois$.  Moreover,
  $\character(\galoiselement(\kernelement))
  =\compconstant_{\galoiselement,-\kernelement}
  \galoiselement(\character(\kernelement))$ for all
  $\kernelement\in\kernfreegroup$.
  
  Consider $(\fieldbar^\times)^{\gradgroupgeneratordiv}$ with the induced
  $\galois$-action, that is,
  \begin{equation*}
    \galoiselement*(\beta_{\freeelement})_{\freeelement\in\gradgroupgeneratordiv}
    \coloneqq (\galoiselement
    (\beta_{\galoiselement^{-1}(\freeelement)}))_{\freeelement\in\gradgroupgeneratordiv}
  \end{equation*}
  for all $\galoiselement\in\galois$ and
  $(\beta_{\freeelement})_{\freeelement\in\gradgroupgeneratordiv}
  \in(\fieldbar^\times)^{\gradgroupgeneratordiv}$. The map
  $\alpha\colon \galois\to(\fieldbar^\times)^{\gradgroupgeneratordiv}$ defined by
  $\alpha(\galoiselement)\coloneqq 
  (\compconstant_{\galoiselement,\galoiselement^{-1}(\freeelement)}
  )_{\freeelement\in\gradgroupgeneratordiv}$
  is a cocycle.  By \cite[\S V.7, Proposition V.8.1]{MR2017446}, there is an
  equivalence of categories between the category of finite sets with a
  continuous $\galois$-action and the category of finite \'etale
  $\field$-algebras. Let $K$ be a finite \'etale $\field$-algebra
  corresponding to $\gradgroupgeneratordiv$. Then
  $(\fieldbar^\times)^{\gradgroupgeneratordiv}$ with the $\galois$-action
  recalled above is the set of $\fieldbar$-rational points of the Weil
  restriction $\weilres_{K/\field}\G_{m}$. By Hilbert's Theorem 90 and the
  Lemma of Eckmann--Faddeev--Shapiro \cite[Theorem 29.2 and Lemma
  29.6]{MR1632779},
  \begin{equation*}
    H^1(\galois,(\fieldbar^\times)^{\gradgroupgeneratordiv})\cong
    H^1_{\text{\it{\'et}}}(\field,\weilres_{K/\field}\G_{m}) \cong
    H^1_{\text{\it{\'et}}}(K,\G_{m})=0.
  \end{equation*}
  Therefore, there is
  $\beta=(\beta_{\freeelement})_{\freeelement\in\gradgroupgeneratordiv}
  \in\gradgroupgeneratordiv\times\fieldbar^\times$
  such that $\alpha(\galoiselement)=\beta^{-1}\galoiselement*\beta$
  for all $\galoiselement\in\galois$. That is,
  $\compconstant_{\galoiselement,\freeelement}=
  \beta_{\galoiselement(\freeelement)}^{-1}\galoiselement(\beta_{\freeelement})$
  for all $\galoiselement\in\galois$ and
  $\freeelement\in\gradgroupgeneratordiv$.  Given
  $\freeelement'\in\freegroup$, write
  $\freeelement'=\sum_{\freeelement\in\gradgroupgeneratordiv}a_{\freeelement}\freeelement$
  with $a_{\freeelement}\in \ZZ$, and define
  \begin{equation*}
    \beta_{\freeelement'}\coloneqq 
    \prod_{\freeelement\in\gradgroupgeneratordiv}\beta_{\freeelement}^{a_{\freeelement}}.
  \end{equation*}
  The homomorphism $\character'\colon \kernfreegroup\to\FfieldXbar^\times$ that
  sends $\kernelement\in\kernfreegroup$ to
  $\beta_{\kernelement}^{-1}\character(\kernelement)$ is $\galois$-equivariant
  and satisfies $\divi_0(\character'((0,\divisor)))=\divisor$ for all
  $(0,\divisor)\in\kernfreegroup$.
\end{proof}

From \cite[Proposition 2.2.8]{MR899402}, we know that the existence of
universal torsors (that is, torsors of type $\identity_{\pic(\Xcbar)}$) for a
\emph{smooth} variety $\Xc$ over a nonclosed field $\field$ is equivalent to
the existence of a $\galois$-equivariant splitting of the exact sequence
\begin{equation}\label{eq:exact_sequence_X_nonclosed}
  1 \to \fieldbar^\times \to \FfieldXbar^\times \to \FfieldXbar^\times/\fieldbar^\times
  \to 1.
\end{equation}
Moreover, if $\opensubset$ is a nonempty open subset of $\Xc$, and
$\gradgroup$ is the kernel of the natural homomorphism
$\pic(\Xcbar)\to\pic(\opensubsetbar)$, the existence of torsors of injective
type $\gradgroup\hookrightarrow\pic(\Xcbar)$ is equivalent to the existence of a
$\galois$-equivariant splitting of the exact sequence
\begin{equation}\label{eq:exact_sequence_U_nonclosed}
  1 \to \fieldbar^\times \to \globalUbar^\times
  \to \globalUbar^\times/\fieldbar^\times \to 1.
\end{equation}
As a consequence of Corollaries \ref{cor:cox_ring_cox_sheaf_nonclosed}
and \ref{cor:bijection_torsors_cox_sheaves}, this is also equivalent
to the existence of Cox sheaves (and Cox rings if
$\gradgroup=\gradgroupeff$) of $\Xc$ of the same type.

In our more general setting, where $\Xc$ can be singular, we show how
to construct a Cox sheaf of $\Xc$ of type $\typecox$ starting from
$\galois$-equivariant splittings of such exact sequences.  Then
Proposition~\ref{prop:cox_sheaves_torsors_nonclosed} gives a torsor of
$\Xc$ of type $\typecox$ in this setting, generalizing the results of
\cite{MR899402} mentioned above.
  
\begin{construction}\label{construction:cox_sheaf_splitting_U}
  Let $\opensubset\subseteq\Xc$ be a nonempty open subset such that
  $\typecox(\gradgroup)$ is contained in the kernel of the natural morphism
  $\pic(\Xcbar)\to\pic(\opensubsetbar)$. Assume that
  \eqref{eq:exact_sequence_U_nonclosed} admits a $\galois$-equivariant
  splitting $\splitting\colon \globalUbar^\times\to\fieldbar^\times$.  For every
  $\groupelement\in\gradgroup$, let $\divisor_\groupelement$ be a Cartier
  divisor supported on $\Xcbar\smallsetminus\opensubsetbar$ representing the
  class $\typecox(\groupelement)$ in $\pic(\Xcbar)$, and let
  $\coxsheaf_{\groupelement}\coloneqq \OO_\Xc(\divisor_\groupelement)$. Endow
  $\coxsheaf\coloneqq \bigoplus_{\groupelement\in\gradgroup}\coxsheaf_{\groupelement}$
  with the following multiplication of sections induced by $\splitting$. For
  every open subset $V$ of $\Xcbar$ and homogeneous sections
  $\coxsec_1,\coxsec_2\in\coxsheaf(V)$ of degree
  $\groupelement_1,\groupelement_2\in\gradgroup$, respectively, let
  $f\in\globalUbar^\times$ be the unique element such that
  $\divisor_{\groupelement_1}+\divisor_{\groupelement_2}
  =\divi_{\divisor_{\groupelement_1+\groupelement_2}}(f)$
  and $\sigma(f)=1$. Define
  $\coxsec_1\coxsec_2\coloneqq f\coxsec_1\coxsec_2
  \in\coxsheaf_{\groupelement_1+\groupelement_2}(V)$,
  where the product on the right is computed in $\FfieldXbar$. Then
  $\coxsheaf$ is a Cox sheaf of $\Xcbar$ of type $\typecox$ that is
  equivariant with respect to the following action of $\galois$. For every
  $\galoiselement\in\galois$ and $\groupelement\in\gradgroup$, let
  $f\in\FfieldXbar^\times$ be the unique element such that
  $\galoiselement\divisor_{\groupelement}
  =\divi_{\divisor_{\galoiselement\groupelement}}(f)$
  and $\splitting(f)=1$. Define
  $\galoiselement*\coxsec\coloneqq f\galoiselement(\coxsec)\in
  \coxsheaf_{\galoiselement\groupelement}$
  for all sections $\coxsec\in\coxsheaf_{\groupelement}$, where the product on
  the right is computed in $\FfieldXbar$.

  Then the sheaf of invariants $\coxsheaf^{\galois}$ is a Cox sheaf of $\Xc$
  of type $\typecox$ by Proposition
  \ref{prop:galois_descent_cox_rings_sheaves}, and $\coxsheaf^{\galois}(\Xc)$
  is a Cox ring of $\Xc$ of type $\typecox$ by Corollary
  \ref{cor:cox_ring_cox_sheaf_nonclosed}.
\end{construction}
  
\begin{remark}\label{rem:cox_sheaf_splitting_X}
  We note that a $\galois$-equivariant splitting of
  \eqref{eq:exact_sequence_X_nonclosed} induces a $\galois$-equivariant
  splitting of \eqref{eq:exact_sequence_U_nonclosed} for all nonempty open
  subsets $\opensubset$ of $\Xc$. Hence a $\galois$-equivariant splitting of
  \eqref{eq:exact_sequence_X_nonclosed} ensures the existence of Cox rings and
  Cox sheaves of $\Xc$ of every type.
\end{remark}
  
The following proposition  generalizes \cite[Proposition 2.2.8(v)]{MR899402}.

\begin{prop}\label{prop:cox_sheaf_splitting_nonclosed}
  If $\Xcbar$ is locally factorial, and $\Xc$ has a Cox sheaf of injective
  type $\typecox$, then the natural exact sequence
  \eqref{eq:exact_sequence_U_nonclosed} has a $\galois$-equivariant splitting
  for every nonempty open subset $\opensubset \subseteq \Xc$ such that
  $\typecox(\gradgroup)$ is the kernel of the natural morphism
  $\pic(\Xcbar)\to\pic(\opensubsetbar)$.

  If in addition $\field$ is perfect and
  $\typecox=\identity_{\pic(\Xcbar)}$, also the exact sequence
  \eqref{eq:exact_sequence_X_nonclosed} admits a $\galois$-equivariant
  splitting.
\end{prop}

\begin{proof}
  Let $\opensubset \subseteq \Xc$ be as in the statement. Then the group
  $\freegroup$ of Cartier divisors on $\Xcbar$ supported outside
  $\opensubsetbar$ is free, finitely generated and has a $\galois$-invariant
  basis (consisting of the prime divisors supported outside
  $\opensubsetbar$). The group $\globalUbar^\times/\fieldbar^\times$ is
  naturally identified with the subgroup $\kernfreegroup\subseteq\freegroup$
  of principal divisors. By Proposition \ref{prop:equivariant_character},
  there exists a $\galois$-equivariant character
  $\character\colon \kernfreegroup\to\FfieldXbar^\times$ associated with
  $\freealgebra\coloneqq \bigoplus_{\divisor\in\freegroup}\OO_\Xc(\divisor)$.  Since
  every element of $\Lambda_0$ is supported outside $\opensubsetbar$, the
  image of $\character$ is contained in $\globalUbar^\times$. Therefore,
  $\character$ is a $\galois$-equivariant splitting of the exact sequence
  \eqref{eq:exact_sequence_U_nonclosed} associated with $\opensubset$.
  
  Assume now that $\field$ is perfect and that
  $\typecox=\identity_{\pic(\Xcbar)}$.  Let $X'$ be the smooth locus of
  $\Xcbar$ and $U'\coloneqq \opensubsetbar\cap X'$. Since $\Xcbar$ is normal, its
  singular locus has codimension $\geq2$. Hence $\FfieldXbar=\fieldbar(X')$
  and $\globalUbar=\fieldbar[U']$.  Then, according to \cite[Proposition
  2.2.8]{MR899402}, the exact sequence \eqref{eq:exact_sequence_X_nonclosed}
  admits a $\galois$-equivariant splitting whenever
  \eqref{eq:exact_sequence_U_nonclosed} does.
\end{proof}
  
Every rational point $x\in\opensubset(\field)$ defines a $\galois$-equivariant
splitting $\sigma_x\colon \globalUbar^\times\to\fieldbar^\times$ of
\eqref{eq:exact_sequence_U_nonclosed} by $\sigma_x(f)\coloneqq f(x)$ for all
$f\in\globalUbar^\times$.  Moreover, every smooth $\field$-rational point on
$\Xc$ defines a $\galois$-equivariant splitting of
\eqref{eq:exact_sequence_X_nonclosed} by \cite[Remarque 2.2.3]{MR899402}.
Therefore, if $\Xc$ has a smooth $\field$-rational point, Cox sheaves and Cox
rings of $\Xc$ of every type exist by Construction
\ref{construction:cox_sheaf_splitting_U}.  If $\Xc(\field)\neq\emptyset$, the
same result can be obtained combining \cite[Proposition 1]{MR0447246} with
Proposition \ref{prop:torsors_cox_sheaves_nonclosed}.
 
The following construction generalizes \cite[Construction 1.4.2.3]{arXiv:1003.4229}.
 
\begin{construction}\label{construction:rational_point_cox_sheaf}
  Let $x\in\Xc(\field)$. Let
  $\gradgroupgenerator, \divisorsgenerator, \gradgroupgeneratordiv,
  \freegroup, \presentation, \freealgebra$ be as in
  Proposition~\ref{prop:equivariant_character}.
  Let $\kernfreegroup$ be the kernel of
  $\presentation$.  The morphism
  $\character\colon \kernfreegroup\to\FfieldXbar^\times$ that sends
  $(0,\divisor)\in\kernfreegroup$ to the unique element
  $f\in\FfieldXbar^\times$ such that $\divi_{0}(f)=\divisor$ and $f(x)=1$ is a
  $\galois$-equivariant character associated with $\freealgebra$. Let
  $\freeideal$ be the sheaf of ideals of $\freealgebra$ defined by
  $\character$ as in Construction \ref{construction:cox_sheaf_abstract}.
  Then the Cox sheaf $\coxsheaf\coloneqq \freealgebra/\freeideal$ is
  $\galois$-equivariant by Proposition \ref{prop:equivariant_character}. Let
  $\torsormor\colon \spec_\Xc\coxsheaf^\galois\to\Xc$ be the induced torsor of type
  $\typecox$. Then $x\in\torsormor((\spec_\Xc\coxsheaf^\galois)(\field))$.
\end{construction}

\begin{proof}
  The set of Cartier divisors of $\Xcbar$ that do not contain $x$ in
  their support form a $\galois$-invariant group that generates
  $\pic(\Xcbar)$. Indeed, if $\divisor=\{(U_i,f_i)\}_i$ is a Cartier
  divisor on $\Xcbar$ that contains $x$ in its support, take $j$ such
  that $x\in U_j$. Then the divisor $\divi_\divisor(f_j^{-1})$ is
  linearly equivalent to $\divisor$ and is supported outside
  $x$. Therefore, it is always possible to choose a set
  $\divisorsgenerator$ as in the statement. The character $\character$
  defined by $x$ as above is $\galois$-equivariant because $x$ is
  $\galois$-invariant.
  
  Let $U$ be an affine open neighborhood of $x$ in $\Xc$. Then the
  homomorphism $\morphism\colon \freealgebra(U_{\fieldbar})\to\fieldbar$
  defined by $\morphism(\coxsec)\coloneqq \coxsec(x)$ for all homogeneous
  sections $\coxsec\in\freealgebra(U_{\fieldbar})$ is well-defined
  because $\divisor$ is supported outside $x$ for all
  $(\groupelement,\divisor)\in\freegroup$, and $\galois$-equivariant
  because $x$ is $\galois$-invariant. Since
  $\character(\kernelement)(x)=1$ for all
  $\kernelement\in\kernfreegroup$, the homomorphism $\morphism$
  factors through $\coxsheaf(U_{\fieldbar})$ and defines a
  $\field$-rational point on $\torsormor^{-1}(x)$ by Galois descent.
\end{proof}

\begin{remark}\label{rem:principal_universal_torsor}
  If $\Xc$ is a smooth complete toric variety containing an open torus
  $U$, then we can apply Construction
  \ref{construction:rational_point_cox_sheaf} with
  $\gradgroup\coloneqq \pic(\Xcbar)$, $\typecox\coloneqq \identity_{\pic(\Xcbar)}$, $x$
  the neutral element of $U$, $\divisorsgenerator$ the set of
  $U$-invariant divisors on $\Xcbar$, and $\gradgroupgenerator$ the set
  of their classes. The resulting Cox sheaf corresponds to \emph{the
    principal universal torsor} described in \cite[\S2.4.4]{MR899402}
  and \cite[end of \S8]{MR1679841}.
\end{remark}

We recall that if $\lambda$ factors through $\pic(\Xc) \hookrightarrow
\pic(\Xcbar)$, a Cox ring of $\Xc$ of type $\lambda$ always exists,
even if $\Xc(\field)=\emptyset$.  The following example shows that the
existence of a $\field$-rational point on $\Xc$ is not necessary for
the existence of Cox rings of $\Xc$ of arbitrary type. See
\cite[Exemples~2.2.12]{MR899402} for further examples.

\begin{example}\label{exa:cox_ring_no_rational_point}
  Let $k$ be an arbitrary number field.  Let $\Xc$ be the smooth
  projective fourfold over $k$ in \cite[Theorem
  3.5]{arXiv:1409.6706}. Then $\Xc(k)=\emptyset$, and $X$ is a
  counterexample to the Hasse principle, but there is no \'etale (and
  hence no algebraic) Brauer--Manin obstruction to the Hasse
  principle. Moreover, $\pic(\Xcbar)$ is a finitely generated abelian
  group, as the Albanese variety of $X$ is trivial. Therefore, $X$ has
  a universal torsor by \cite[Corollary 6.1.3]{MR1845760}, and Cox
  rings and Cox sheaves of $\Xc$ of all types exist by
  Proposition~\ref{prop:torsors_cox_sheaves_nonclosed} and Remark
  \ref{rem:pullback_short}.
\end{example}

\begin{proof}[Proof of Theorem \ref{thm:splttings}]
  For the statements regarding Cox sheaves, see Proposition
  \ref{prop:equivariant_character}, Construction
  \ref{construction:cox_sheaf_splitting_U}, Remark
  \ref{rem:cox_sheaf_splitting_X}, Proposition
  \ref{prop:cox_sheaf_splitting_nonclosed} and Construction
  \ref{construction:rational_point_cox_sheaf}. For the statements regarding
  torsors and Cox rings, we combine these with Theorem
  \ref{introduction:theorem:classification}.
\end{proof}

\section{Finitely generated Cox rings}\label{section:finitely_generated_cox_rings}

We observe that, for every Cox sheaf $\coxsheaf$ of $\Xc$ of type $\typecox$,
the natural morphism $\spec_{\Xc}\coxsheaf\to\spec\coxsheaf(\Xc)$ is
equivariant under the $\gradgroupdual$-actions on $\spec_{\Xc}\coxsheaf$ and
$\spec\coxsheaf(\Xc)$ induced by the $\gradgroup$-grading on $\coxsheaf$ and
$\coxsheaf(\Xc)$, respectively.  We show that, if $\coxsheaf(X)$ is finitely
generated as a $k$-algebra, this morphism can be used to realize the torsor
$\spec_\Xc\coxsheaf$ as a quasiaffine variety, as in \cite[Construction
1.6.3.1]{arXiv:1003.4229} and \cite[Proposition 3.10]{MR1995130}.

\begin{prop}\label{prop:open_immersion}
  Assume that $\field=\fieldbar$, that $\coxsheaf$ is a Cox sheaf of
  $\Xc$ of type $\typecox$ such that $\coxsheaf(\Xc)$ is finitely
  generated as a $\field$-algebra, and that there are nonzero
  homogeneous sections $f_1,\dots,f_t\in\coxsheaf(\Xc)$ of degrees
  $\groupelement_1,\dots,\groupelement_t\in\gradgroup$, respectively,
  such that the open subsets
  $\Xc\smallsetminus\Supp(\divi_{\divisor_i}(\coxisom_{\groupelement_i,\divisor_i}(f_i)))$
  are affine and cover $\Xc$.

  Then the natural morphism
  $\spec_\Xc\coxsheaf\to\spec\coxsheaf(\Xc)$ is an
  $\gradgroupdual$-equivariant open immersion, and the complement of
  the image is defined by the ideal $\sqrt{(f_1,\dots,f_t)}$ of
  $\coxsheaf(\Xc)$.
\end{prop}
 
\begin{proof}
  Let $\torsormor\colon \spec_\Xc\coxsheaf\to \Xc$ be the morphism induced by
  $\OO_\Xc\subseteq\coxsheaf$. The open subsets
  $\torsormor^{-1}(\Xc\smallsetminus
  \Supp(\divi_{\divisor_i}(\coxisom_{\groupelement_i,\divisor_i}(f_i)))$ are
  affine and cover $\spec_\Xc\coxsheaf$. Moreover,
  \begin{equation*}
    \coxsheaf(\Xc\smallsetminus
    \Supp(\divi_{\divisor_i}(\coxisom_{\groupelement_i,\divisor_i}(f_i)))
    =\coxsheaf(\Xc)[f_i^{-1}]
  \end{equation*}
  for all $i\in\{1,\dots,t\}$. Hence
  $\spec_\Xc\coxsheaf\to\spec\coxsheaf(\Xc)$ is an open immersion whose image
  is the union of the principal open subsets of $\spec\coxsheaf(\Xc)$ defined
  by $f_i$ for $i\in\{1,\dots,t\}$.
\end{proof}

\begin{remark}
  The assumption of Proposition \ref{prop:open_immersion} on the
  affine open covering is equivalent to the requirement that there are
  effective Cartier divisors $\divisor_1,\dots,\divisor_t$ on $\Xc$
  such that $[\divisor_i]\in\typecox(\gradgroup)$ for all
  $i\in\{1,\dots,t\}$ and such that the open subsets
  $\Xc\smallsetminus\Supp(\divisor_i)$ are affine and cover $\Xc$.  If
  $\typecox(\gradgroup)=\pic(\Xc)$, this is the definition of
  divisorial variety \cite[\S3]{MR0153683}. Among those there are
  quasi-projective varieties and locally factorial varieties
  \cite[\S4]{MR0153683}.
\end{remark}

If $\typecox(\gradgroup)$ contains the class of an ample invertible
sheaf on $\Xc$, then the hypothesis of Proposition
\ref{prop:open_immersion} on the affine open covering is satisfied,
and the open immersion $\spec_\Xc\coxsheaf\to\spec\coxsheaf(\Xc)$ can
be characterized as follows (cf.~\cite[Corollary
1.6.3.6]{arXiv:1003.4229}).

\begin{cor}\label{cor:irrelevant_ideal_projective}
  Assume that $\field=\fieldbar$, that $\Xc$ is projective and has a
  Cox sheaf $\coxsheaf$ of type $\typecox$ such that $\coxsheaf(\Xc)$
  is a finitely generated $\field$-algebra, and there is
  $\groupelement\in\gradgroup$ such that $\typecox(\groupelement)$ is
  very ample. Then $\spec_\Xc\coxsheaf\to\spec\coxsheaf(\Xc)$ is a
  $\gradgroupdual$-equivariant open immersion and the complement of
  the image is defined by the ideal
  $\sqrt{\langle\coxsheaf(\Xc)_{\groupelement}\rangle}$ of
  $\coxsheaf(\Xc)$, where
  $\langle\coxsheaf(\Xc)_{\groupelement}\rangle$ is the ideal
  generated by the degree-$\groupelement$-part of $\coxsheaf(\Xc)$.
\end{cor}

If the Cox sheaf is defined over a nonclosed field $\field$, the open
immersions in Proposition \ref{prop:open_immersion} and Corollary
\ref{cor:irrelevant_ideal_projective} are $\galois$-equivariant, and
hence defined over $\field$.

Now we explain how to realize
a finitely generated Cox ring as a quotient of a polynomial
ring. Without loss of generality, we may assume that
$\gradgroup=\gradgroupeff$ (cf. Remark \ref{rem:effective_degrees}).

\begin{prop}\label{prop:cox_ring_as_polynomial_ring}
  Assume that $\field=\fieldbar$.  Let
  $\groupelement_1,\dots,\groupelement_N\in\gradgroup$ be generators of
  $\gradgroupeff$. Let $\freegroup\coloneqq \bigoplus_{i=1}^N\ZZ\groupelement_i$, and
  let $\kernfreegroup$ be the kernel of the natural homomorphism
  $\presentation\colon \freegroup\to\gradgroup$.  For $i \in \{1, \dots, N\}$, let
  $\divisor_i$ be a Cartier divisor representing the class
  $\typecox(\groupelement_i)$ in $\pic(\Xc)$, and for every
  $\freeelement=\sum_{i=1}^Na_i\groupelement_i$ of $\freegroup$, let
  $\divisor_{\freeelement}\coloneqq \sum_{i=1}^Na_i\divisor_i$.  Let
  $\character\colon \kernfreegroup\to\FfieldX^\times$ be a character associated with
  the $\OO_\Xc$-algebra
  $\freealgebra\coloneqq \bigoplus_{\freeelement\in\freegroup}\OO_\Xc(\divisor_{\freeelement})$.

  Endow $\freering\coloneqq \field[\coxcoord_1,\dots,\coxcoord_N]$ with the
  $\gradgroup$-grading induced by assigning degree $\groupelement_i$
  to $\coxcoord_i$ for each $i\in\{1,\dots,N\}$.  For every
  $\freeelement\in\freegroup$, let
  \begin{equation*}
    \coxisom_{\freeelement}\colon \freering_{\presentation(\freeelement)}\to
    H^0(\Xc,\OO_\Xc(\divisor_{\freeelement}))
  \end{equation*}
  be the linear map that sends
  $\coxcoord_1^{a_1}\cdots\coxcoord_N^{a_N}$ to $
  \character(\sum_{i=1}^Na_i\groupelement_i-\freeelement) $ for all
  nonnegative integers $a_1,\dots,a_N$ such that
  $\sum_{i=1}^Na_i\groupelement_i=\presentation(\freeelement)$ in
  $\gradgroup$.  Let $g_1,\dots,g_s\in\freering$ be homogeneous
  elements such that, for each $i\in\{1,\dots,s\}$, if $g_i$ has
  degree $\presentation(\freeelement_i)$ with
  $\freeelement_i\in\freegroup$, then
  $\coxisom_{\freeelement_i}(g_i)=0$.

  Then $\coxring\coloneqq \freering/(g_1,\dots,g_s)$ is a Cox ring of $\Xc$ of
  type $\typecox$ if and only if the linear map
  \begin{equation*}
    \coxisom_{\presentation(\freeelement),\divisor_{\freeelement}}
    \colon \coxring_{\presentation(\freeelement)}\to H^0(\Xc,\OO_\Xc(\divisor_{\freeelement}))
  \end{equation*} 
  induced by $\coxisom_{\freeelement}$ is an isomorphism for all
  $\freeelement\in\freegroup$.

  Conversely, if $\coxring$ is a finitely generated Cox ring of $\Xc$
  of type $\typecox$, then there are generators
  $\groupelement_1,\dots,\groupelement_N$ of $\gradgroupeff$, a
  character $\character$ and polynomials
  $g_1,\dots,g_s\in\field[\coxcoord_1,\dots,\coxcoord_N]$ as above
  such that $\coxring\cong
  \field[\coxcoord_1,\dots,\coxcoord_N]/(g_1,\dots,g_s)$.
\end{prop}

\begin{proof}
  For the first statement, we notice that, for all
  $\freeelement_1,\freeelement_2\in\freegroup$ and all
  $\coxsec_1,\coxsec_2\in\freering$ homogeneous of degrees
  $\presentation(\freeelement_1), \presentation(\freeelement_2)$,
  respectively,
  \begin{equation*}
    \coxisom_{\freeelement_1}(\coxsec_1)\coxisom_{\freeelement_2}(\coxsec_2)
    =\coxisom_{\freeelement_1+\freeelement_2}(\coxsec_1\coxsec_2)
  \end{equation*} 
  because $\character$ is a group homomorphism.

  Conversely, assume that $\coxring$ is a finitely generated Cox ring
  of $\Xc$ of type $\typecox$, and let $\coxsec_1,\dots,\coxsec_N$ be
  a finite set of homogeneous elements that generate $\coxring$. For
  every $i\in\{1,\dots,N\}$, let $\groupelement_i$ be the degree of
  $\coxsec_i$.  Sending $\coxcoord_i\mapsto\coxsec_i$ defines a
  surjective homomorphism
  $\freemor\colon \field[\coxcoord_1,\dots,\coxcoord_N]\to\coxring$ of
  $\gradgroup$-graded rings, where the grading on
  $\freering\coloneqq \field[\coxcoord_1,\dots,\coxcoord_N]$ is defined by
  assigning degree $\groupelement_i$ to $\coxcoord_i$ for all
  $i\in\{1,\dots,N\}$. Since $\freering$ is noetherian, the kernel of
  $\freemor$ is generated by finitely many homogeneous elements
  $g_1,\dots,g_s$.

  Since $\coxsec_1,\dots,\coxsec_N$ generate $\coxring$, the elements
  $\groupelement_1,\dots,\groupelement_N$ generate $\gradgroupeff$.
  Let $\freegroup\coloneqq \bigoplus_{i=1}^N\ZZ\groupelement_i$, and let
  $\kernfreegroup$ be the kernel of the natural homomorphism
  $\presentation\colon \freegroup\to\gradgroup$.  For every $i\in\{1,\dots,
  N\}$, let $\divisor_i\coloneqq \divi(\coxsec_i)$, and, for every
  $\freeelement=\sum_{i=1}^Na_i\groupelement_i$ in $\freegroup$, let
  $\divisor_{\freeelement}\coloneqq \sum_{i=1}^Na_i\divisor_i$.

  Let
  $\{\coxisom_{\groupelement,\divisor}\}_{(\groupelement,\divisor)\in\gradgroupdiv}$
  be a family of isomorphisms associated with $\coxring$ as in Proposition
  \ref{prop:equivalent_definition_cox_ring_div}.  Let
  $\freegroup_+$ be the monoid generated by $\groupelement_1, \dots, \groupelement_N$.
  For every
  $\freeelement=\sum_{i=1}^Na_i\groupelement_i\in\freegroup_+$, let
  $\compconstant_{\freeelement}\coloneqq \coxisom_{\presentation(\freeelement),\divisor_\freeelement}
  (\coxsec_1^{a_1}\cdots\coxsec_N^{a_N})$.
  For every $\freeelement\in\freegroup$ such that
  $\typecox(\presentation(\freeelement))$ is an effective class, write
  $\freeelement=\freeelement^+-\freeelement^-$ with
  $\freeelement^+,\freeelement^-\in\freegroup_+$, and define
  $\compconstant_{\freeelement}\coloneqq \compconstant_{\freeelement^+}
  \compconstant_{\freeelement^-}^{-1}\compconstant$,
  where $\compconstant\in\field^\times$ is the unique constant with
  $\coxisom_{\presentation(\freeelement),\divisor_{\freeelement}}(\coxsec)
  \coxisom_{\presentation(\freeelement^-),\divisor_{\freeelement^-}}(\coxsec')
  =\compconstant\coxisom_{\presentation(\freeelement^+),\divisor_{\freeelement^+}}
  (\coxsec\coxsec')$ for all
  $\coxsec\in \coxring_{\presentation(\freeelement)}$ and
  $\coxsec'\in\coxring_{\presentation(\freeelement^-)}$. The constant
  $\compconstant_{\freeelement}$ does not depend on the choice of
  $\freeelement^+$ and $\freeelement^-$. If
  $\typecox(\presentation(\freeelement))$ is not effective, take
  $\compconstant_\freeelement\coloneqq 1$.

  For every $\freeelement\in\freegroup$, let
  $\coxisom_{\freeelement}\colon \freering_{\presentation(\freeelement)}\to
  H^0(\Xc,\OO_\Xc(\divisor_{\freeelement}))$ be the linear map that sends
  $\coxcoord_1^{a_1}\cdots\coxcoord_N^{a_N}$ to
  $ \compconstant_{\freeelement}^{-1}
  \coxisom_{\presentation(\freeelement),\divisor_\freeelement}
  (\coxsec_1^{a_1}\cdots\coxsec_N^{a_N})$ for all
  $(a_1,\dots,a_N)\in\ZZ_{\geq0}^N$ such that
  $\sum_{i=1}^Na_i\groupelement_i=\presentation(\freeelement)$ in
  $\gradgroup$.  The morphisms $\coxisom_{\freeelement}$ satisfy
  \begin{equation*}
    \coxisom_{\freeelement}(\coxsec)\coxisom_{\freeelement'}(\coxsec')
    =\coxisom_{\freeelement+\freeelement'}(\coxsec\coxsec')
  \end{equation*} 
  for all $\coxsec,\coxsec'\in\freering$ homogeneous of degrees
  $\presentation(\freeelement),\presentation(\freeelement')$,
  respectively, for all
  $\freeelement,\freeelement'\in\freegroup$. Then the map
  $\character\colon \kernfreegroup\to\FfieldX^\times$ defined by
  $\chi(\kernelement)\coloneqq \coxisom_{-\kernelement}(1)$, for all
  $\kernelement\in\kernfreegroup$, is a character associated with
  $\bigoplus_{\freeelement\in\freegroup}\OO_\Xc(\divisor_\freeelement)$
  that defines on the $\field$-algebra
  $\field[\coxcoord_1,\dots,\coxcoord_N]/(g_1,\dots,g_s)$ a structure
  of Cox ring of $\Xc$ of type $\typecox$, which coincides with the
  one induced by $\coxring$ via $\freemor$.
\end{proof}

\begin{remark}\label{rem:cox_ring_as_polynomial_ring_nonclosed}
  Let $\field$ be an arbitrary field, and consider the construction in
  Proposition \ref{prop:cox_ring_as_polynomial_ring} for $\Xcbar$.  If the set
  $\{(\groupelement_1,\divisor_1),\dots,(\groupelement_N,\divisor_N)\}$ is
  $\galois$-invariant with respect to the componentwise action of $\galois$
  and if the character $\character$ is $\galois$-equivariant, then
  $\coxisom_{\galoiselement\freeelement}\circ\galoiselement
  =\galoiselement\circ\coxisom_{\freeelement}$
  for all $\galoiselement\in\galois$ and all $\freeelement\in\freegroup$,
  where $\{\coxisom_{\freeelement}\}_{\freeelement\in\freegroup}$ are the
  morphisms defined by $\character$ and $\galois$ acts on $\freering$ by
  permuting the generators as follows:
  $\galoiselement(\coxcoord_i)=\coxcoord_j$ if
  $\galoiselement(\groupelement_i)=\groupelement_j$.  Therefore, if
  $\coxring\coloneqq \freering/(g_1,\dots,g_s)$ is a Cox ring of $\Xcbar$ of type
  $\typecox$, it is endowed with a {\compatible} $\galois$-action, and
  descends to a Cox ring $\coxring^\galois$ of $\Xc$ of type $\typecox$ by
  Proposition \ref{prop:galois_descent_cox_rings_sheaves}. The ring
  $\coxring^\galois$ is a finitely generated $\field$-algebra by faithfully
  flat descent.
\end{remark}

We conclude this section by proving that finite generation is preserved under
pullback of Cox rings.

\begin{prop}
  Let $\groupmorphism\colon \gradgroup'\to\gradgroup$ be a homomorphism of
  finitely generated abelian groups, and $\coxring$ a finitely
  generated $\gradgroup$-graded $\field$-algebra.  Then
  $\groupmorphism^*\coxring\coloneqq 
  \bigoplus_{\groupelement'\in\gradgroup'}\coxring_{\groupmorphism(\groupelement')}$
  is finitely generated as a $\field$-algebra.
\end{prop}

\begin{proof}
  If $\groupmorphism$ is injective, the statement holds by
  \cite[Proposition 1.1.2.4]{arXiv:1003.4229}. Hence, by factoring
  $\groupmorphism$ through the inclusion
  $\groupmorphism(\gradgroup')\subseteq\gradgroup$, we can assume
  without loss of generality that $\groupmorphism$ is surjective.

  Fix homogeneous generators $\coxsec_1,\dots,\coxsec_n$ of $\coxring$
  of degrees $\groupelement_1,\dots,\groupelement_n\in\gradgroup$,
  respectively, and let $\groupelement'_1,\dots,\groupelement'_{n'}$
  be generators for $\gradgroup'$.  Up to enlarging the sets
  $(\coxsec_1,\dots,\coxsec_n)$ and
  $(\groupelement'_1,\dots,\groupelement'_{n'})$ and permuting the
  indices, we can assume that $n=n'$ and that
  $\groupmorphism(\groupelement'_{i})=\groupelement_i$ for all
  $i\in\{1,\dots n\}$.  Let $\gradgroup_0$ be the kernel of
  $\groupmorphism$, and fix elements
  $\groupelement''_1,\dots,\groupelement''_{r}\in\gradgroup_0$ that
  generate $\gradgroup_0$ as a monoid.  Let
  \begin{equation*}
    \morphism\colon  \field[t_{1},\dots,t_{n}, u_1,\dots,u_r]\to
    \bigoplus_{\groupelement'\in\gradgroup'}\coxring_{\groupmorphism(\groupelement')}
    \eqqcolon \groupmorphism^*\coxring
  \end{equation*}
  be the morphism that sends the variable $t_{i}$ to the section $\coxsec_{i}$
  in degree $\groupelement'_{i}$, for all $i \in \{1, \dots, n\}$, and the
  variable $u_j$ to the section $1$ in degree $\groupelement''_j$, for all
  $j \in \{1, \dots, r\}$.

  We show that $\morphism$ is surjective. Let $\coxsec$ be a
  homogeneous element of $\groupmorphism^*\coxring$ of arbitrary
  degree $\groupelement'\in\gradgroup'$. Since
  $\coxsec\in\coxring_{\groupmorphism(\groupelement')}$ and $\coxring$
  is generated by $\coxsec_1,\dots,\coxsec_n$ as a $\field$-algebra,
  we can assume, without loss of generality, that $\coxsec$ is a
  monomial in the generators, i.e.,
  $\coxsec=\prod_{i=1}^n\coxsec_i^{a_i}$ for suitable nonnegative
  integers $a_1,\dots,a_n$. Then the element
  $\groupelement''\coloneqq \groupelement'-\sum_{i=1}^na_i\groupelement'_i$
  is in $\gradgroup_0$. Write
  $\groupelement''=\sum_{i=1}^rb_i\groupelement''_i$ for suitable
  nonnegative integers $b_1,\dots,b_r$. Then
  $\coxsec=\morphism((\prod_{i=1}^nt_i^{a_i})(\prod_{j=1}^ru_j^{b_j}))$.
\end{proof}

The pullback of a finitely generated $\field$-algebra under an
injective morphism of grading groups is also called a \emph{Veronese
  subalgebra}; see \cite{arXiv:1003.4229} for example.  Generators and
relations of Veronese subalgebras can be computed via an algorithm
provided by the Maple package for Mori dream spaces developed by
Hausen and Keicher \cite{software_Cox_rings}.

\section{Arithmetic applications}\label{sec:applications}

The embedding from Proposition \ref{prop:open_immersion} provides an explicit
description of torsors as open subsets of closed subsets of affine
spaces. Such a description of universal torsors has been used to prove Manin's
conjecture on the distribution of rational points for certain varieties over
number fields; see
\cite{MR1679841,MR1909606,MR2332351,MR3198752,arXiv:1312.6603}, for
example. In these cases, the varieties are \emph{split} (see
Section~\ref{sec:motivation}).

If a variety is not split,
other torsors may give more suitable parameterizations for the
purpose of counting rational points.  We present three examples where the
parameterizations for the study of rational points on \emph{nonsplit}
varieties induced by torsors of \emph{non-identity type} can be determined by
computing the associated Cox rings.

\begin{example}
  See \cite{MR2373960} for a proof of Manin's conjecture for the nonsplit
  singular quartic del Pezzo surface $S \subset \PP^4_\QQ$ defined by
  \begin{equation*}
    \dcoordX_0\dcoordX_1-\dcoordX_2^2
    =\dcoordX_0^2-\dcoordX_1\dcoordX_4+\dcoordX_3^2=0.
  \end{equation*}
  Over $\QQ(\imaginary)$, it contains exactly one singular point of type
  $\Dfour$ in $(0:0:0:0:1)$ and two conjugate lines.
  Let $X$ be a minimal desingularization of $S$. A central step in this proof
  of Manin's conjecture is the parameterization of $X(\QQ)$ in \cite[\S
  3]{MR2373960}, which is obtained by a sequence of elementary manipulations
  of the defining equations.
  
  One can show that this parameterization of $X(\QQ)$ is induced by a
  $\ZZ$-model of a torsor of type $\pic(X)\hookrightarrow\pic(\Xcbar)$. To prove it,
  one can start with the Cox ring of $X_{\overline{\QQ}}$ of identity type
  computed in \cite[\S3.4]{math.AG/0604194} and use the theory developed in
  our present work together with the MDS package \cite{software_Cox_rings} and
  results on models of torsors from \cite[\S 1, \S 3]{Pieropan_thesis}.
\end{example}

\begin{example}
  A \emph{Ch\^atelet surface} $\Xc$ over a
  field $\field$ is a smooth compactification of an affine surface
  defined in $\A^3_\field$ by an equation of the form
  \begin{equation*}
    \chatcoordX^2-a\chatcoordY^2=\chatpoly(\chatcoordZ),
  \end{equation*}
  where $\chatpoly$ is a separable polynomial of degree
  $4$ and $a \in \field^\times$; see \cite{MR870307,MR876222}.

  To determine the Cox rings of $X$, one can start with the Cox ring 
  $\chatcoxbar$ of $X_{\fieldbar}$ of identity
  type from \cite[6.4(iii)]{thesis}, descend it to a Cox ring
  $\chatcox$ of $X$ of type $\identity_{\pic(\chatbar)}$ and compute a set of
  twists $\chatcox^\galoiscocycle$ representing all isomorphism classes of such
  Cox rings completely explicitly. Furthermore, one can compute a Cox ring of
  $X$ of type $\pic(\chat)\hookrightarrow\pic(\chatbar)$.
  
  If $a=-1$ and $\chatpoly$ splits over $\field$, this gives another explicit
  construction of the universal torsors and the split torsor described in
  \cite[\S 4]{MR2874644}, which are used in the proof of Manin's conjecture
  for such $X$ over $\field=\QQ$ \cite[Theorem~3.3]{MR2874644}. See
  \cite[\S2.4.3]{Pieropan_thesis} for the details.

  Also for other families of Ch\^atelet surfaces, Cox rings and torsors of
  various types appear in proofs of Manin's conjecture. Based on our
  techniques, they are explicitly computed in \cite[\S 7.2--\S
  7.3]{arXiv:1509.07060} and \cite[\S 4--\S 5]{arXiv:1711.01882}.
\end{example}

\begin{example}\label{exa:quintic_dp}
  In the remainder of this section, we determine Cox rings of type
  $\pic(\Xc)\hookrightarrow\pic(\Xcbar)$ for a smooth quintic del Pezzo surface $X$
  that is a blow-up of the projective plane over a field $\field$ in four
  conjugate points.
\end{example}

\begin{prop}\label{prop:quintic_dp}
  Let $\field$ be a field with separable closure $\fieldbar$. Let
  $\pi \colon  X \to \PP^2_\field$ be a blow-up of four $\fieldbar$-points in
  general position that form one orbit under the
  $\gal(\fieldbar/\field)$-action.  Let $Q_1,Q_2 \in \field[x_0,x_1,x_2]$ be
  non-proportional quadratic forms vanishing in the blown-up points.

  Then a Cox ring of $\Xc$ of type
  $\Pic(\Xc)\hookrightarrow\pic(\Xcbar)$ is given by the
  $\field$-algebra
  \begin{equation*}
    R=\field[\dcoordfield_0,\dots,\dcoordfield_5]/(
    Q_1(\dcoordfield_0,\dcoordfield_1,\dcoordfield_2)-\dcoordfield_3\dcoordfield_4,
    Q_2(\dcoordfield_0,\dcoordfield_1,\dcoordfield_2)-\dcoordfield_3\dcoordfield_5),
  \end{equation*}
  where $\divi(\dcoordfield_0),\divi(\dcoordfield_1),\divi(\dcoordfield_2)$
  are the total transforms of the three coordinate lines in $\PP^2_\field$,
  $\divi(\dcoordfield_3)$ is the exceptional divisor of $\pi$, and
  $\divi(\dcoordfield_4),\divi(\dcoordfield_5)$ are the strict transforms of
  the conics defined by $Q_1,Q_2$, respectively.
\end{prop}

\begin{proof}
  Our starting point is the well-known description of Cox rings of identity
  type on $\Xcbar$ by the Pl\"ucker equations of $\mathrm{Gr}(2,5)$; see
  \cite[Examples 3.3.4 and 4.2.4]{MR1679842}, for example.

  Let $E$ be the exceptional divisor of $\pi$. Let $\ell_0,\dots, \ell_4$ be
  the basis of $\Pic(\Xcbar)$ where $\ell_0$ is the pullback of a hyperplane
  class in $\PP^2_\field$ and $\ell_1,\dots,\ell_4$ are the classes of the
  four exceptional divisors of $\pi$ over $\kbar$. Then $\ell_0$ and
  $[E] = \ell_1+\ell_2+\ell_3+\ell_4$ are a basis of $\Pic(\Xc)$. The
  generators $\xi_i$ with $D_i\coloneqq \divi(\xi_i)$ of a Cox ring of type
  $\Pic(\Xc)\hookrightarrow\pic(\Xcbar)$ are computed by hand or using the MDS package
  \cite{software_Cox_rings}.

  Let $\Lambda = \bigoplus_{i=0}^5 \ZZ \ddiv_i$, let
  $\varphi\colon  \Lambda \to \Pic(\Xcbar)$ send $D_i$ to $[D_i]$, and let
  $\Lambda_0=\ker(\varphi)$. Using the character
  $\chi\colon  \Lambda_0 \to \field(\Xc)^\times$ defined by the conditions
  $\chi(D_i-D_0)=\pi^*(x_i/x_0)$ for $i \in \{1,2\}$ and
  $\chi(D_i-(2D_0-D_3))=\pi^*(Q_{i-3}(x_0,x_1,x_2)/x_0^2)$ for $i \in \{4,5\}$,
  we
  find our two relations in degree $2\ell_0$. The MDS package shows that there
  are no further relations.
\end{proof}

In the situation of Proposition~\ref{prop:quintic_dp}, let $-K$ be the
anticanonical class on $X$.  We observe that $R_{-K}$ is the $k$-vector space
with basis $\{\xi_i\xi_j : i\in\{0,1,2\}, j\in\{4,5\}\}$.  Since $-K$ is very
ample, a torsor of type $\pic(X)\hookrightarrow\pic(\Xbar)$ can be realized as
$t\colon \spec R\smallsetminus V(R_{-K})\to X$ by
Corollary~\ref{cor:irrelevant_ideal_projective}.
  
Let $\field$ be a number field and $\OO_\field$ its ring of integers. Let
$\mathscr{Y}\coloneqq \spec\mathscr{R}\smallsetminus V(\mathscr{R}_{-K})$ with
\begin{equation*}
  \mathscr{R}\coloneqq \OO_\field[\dcoordfield_0,\dots,\dcoordfield_5]/(
  Q_1(\dcoordfield_0,\dcoordfield_1,\dcoordfield_2)-\dcoordfield_3\dcoordfield_4,
  Q_2(\dcoordfield_0,\dcoordfield_1,\dcoordfield_2)-\dcoordfield_3\dcoordfield_5).
\end{equation*}
Applying 
\cite[Theorem 3.3]{Pieropan_thesis} to
$(\spec\mathscr{R};\{\xi_i\xi_j : i\in\{0,1,2\}, j\in\{4,5\}\})$, the
$\OO_k$-model $\mathscr{Y}\to\mathscr{X}$ of $t$ provided by
\cite[Construction 3.1]{Pieropan_thesis} is a torsor under $\G_{m,\OO_k}^2$.

For $k=\QQ$, this gives is a $(4:1)$-map from
\begin{equation*}
  \mathscr{Y}(\ZZ) = \left\{(\xi_0,\dots,\xi_5)\in\ZZ^6:
    \begin{aligned}
      &\gcd(\dcoordfield_0,\dcoordfield_1,\dcoordfield_2)=1,\
      \gcd(\dcoordfield_4,\dcoordfield_5)=1\\
      &Q_1(\dcoordfield_0,\dcoordfield_1,\dcoordfield_2)=\dcoordfield_3\dcoordfield_4,\ 
      Q_2(\dcoordfield_0,\dcoordfield_1,\dcoordfield_2)=\dcoordfield_3\dcoordfield_5
    \end{aligned}
  \right\}
\end{equation*}
to $X(\QQ)$.

\begin{remark}[D.~Loughran]
  Consider the embedding $X \subset \PP^2_\field \times \PP^1_\field$ that we
  obtain by combining $\pi$ with the conic bundle structure
  $X \to \PP^1_\field$ corresponding to $Q_1,Q_2$. Then $X$ is defined by the
  equation
  $\xi_4\cdot Q_2(\xi_0,\xi_1,\xi_2)=\xi_5\cdot Q_1(\xi_0,\xi_1,\xi_2)$ in
  $\PP^2_\field \times \PP^1_\field$ with coordinates
  $((\xi_0:\xi_1:\xi_2),(\xi_4:\xi_5))$.  Over $\field = \QQ$, choosing
  primitive integral coordinates and defining
  $\xi_3=\gcd(Q_1(\xi_0,\xi_1,\xi_2),Q_2(\xi_0,\xi_1,\xi_2))$ leads to our
  parameterization.
\end{remark}

\bibliographystyle{alpha}

\bibliography{generalized_cox_rings_4}

\end{document}